\documentclass[11pt,reqno]{amsart}
\usepackage{amssymb}
\usepackage{amsmath}
\usepackage{amsthm}
\usepackage{graphicx}
\usepackage{caption}
\usepackage{bm}
\usepackage[colorlinks=true,urlcolor=blue,
citecolor=red,linkcolor=blue,linktocpage,pdfpagelabels,
bookmarksnumbered,bookmarksopen]{hyperref}
\usepackage[titletoc,title]{appendix}
\numberwithin{equation}{section} 
\newcounter{cont}[section] 

\newtheorem{thm}[cont]{Theorem}
\newtheorem{prop}[cont]{Proposition}
\newtheorem{lem}[cont]{Lemma}
\newtheorem{cor}[cont]{Corollary}
\theoremstyle{definition}

 \theoremstyle{remark}
 \newtheorem{rem}[cont]{Remark}

\newcommand{\N}{\mathbb{N}}
\newcommand{\R}{\mathbb{R}}
\newcommand{\e}{\varepsilon}

\begin{document}
\baselineskip=16pt

\title[Reaction-diffusion models with Perona--Malik diffusion]{Layered patterns in reaction-diffusion models with Perona--Malik diffusions}

\author[A. De Luca]{Alessandra De Luca}
\address[Alessandra De Luca]{Dipartimento di Scienze Molecolari e Nanosistemi, Universit\`a Ca' Foscari Venezia Mestre,
Campus Scientifico, Via Torino 155, 30170 Venezia Mestre (Italy)}
\email{alessandra.deluca@unive.it}

\author[R. Folino]{Raffaele Folino}
\address[Raffaele Folino]{Departamento de Matem\'aticas y Mec\'anica, Instituto de Investigaciones en Matem\'aticas Aplicadas y en Sistemas, Universidad Nacional Aut\'onoma de M\'exico, Circuito Escolar s/n, C.P. 04510 Cd. de M\'exico (M\'exico)}
\email{folino@mym.iimas.unam.mx}
 
\author[M. Strani]{Marta Strani}
\address[Marta Strani]{Dipartimento di Scienze Molecolari e Nanosistemi, Universit\`a Ca' Foscari Venezia Mestre,
Campus Scientifico, Via Torino 155, 30170 Venezia Mestre (Italy)}
\email{marta.strani@unive.it}


\maketitle


\begin{abstract} 
In this paper we deal with a reaction-diffusion equation in a bounded interval of the real line with a nonlinear diffusion of 
Perona--Malik's type and a balanced bistable reaction term. 
Under very general assumptions, we study the persistence of layered solutions, showing that it strongly depends on the behavior of the reaction term close to the stable equilibria $\pm1$, described by a parameter $\theta>1$.
If $\theta\in(1,2)$, we prove existence of steady states oscillating (and touching) $\pm1$, called \emph{compactons}, 
while in the case $\theta=2$ we prove  the presence of \emph{metastable solutions}, namely solutions with a transition layer structure which is maintained for an exponentially long time. 
Finally, for $\theta>2$, solutions with an unstable transition layer structure persist only for an algebraically long time.  
\end{abstract}

\section{Introduction}\label{sec:intro}
The goal of this paper is to investigate the persistence of phase transition layer solutions to the reaction-diffusion equation
\begin{equation}\label{eq:Q-model}
	u_t=Q(\e^2u_x)_x-F'(u),
\end{equation}
where $u=u(x,t) : [a,b]\times(0,+\infty)\rightarrow \R$, complemented with homogeneous Neumann boundary conditions
\begin{equation}\label{eq:Neu}
	u_x(a,t)=u_x(b,t)=0, \qquad \qquad t>0,
\end{equation}
and initial datum
\begin{equation}\label{eq:initial}
	u(x,0)=u_0(x), \qquad \qquad x\in[a,b].
\end{equation}
In \eqref{eq:Q-model} $\e>0$ is a small parameter,  $Q:\R\to\R$ is a Perona--Malik's type diffusion \cite{PerMal},
while $F:\R\to\R$ is a double well potential with wells of equal depth.
More precisely, we assume that $Q\in C^1(\R)$ satisfies
\begin{equation}\label{eq:Q-ass1}
	\lim_{s\to\pm\infty}Q(s)=Q(0)= 0, \qquad \qquad Q(-s)=-Q(s),
\end{equation}
for all $s\in\R$ and that there exists $\kappa>0$ such that
\begin{equation}\label{eq:Q-ass2}
	Q'(s)>0, \quad \forall \, s \in (-\kappa,\kappa) \qquad \mbox{ and } \qquad Q'(s)<0, \quad \mbox{ if } |s|>\kappa.
\end{equation}
The prototype examples we have in mind are
\begin{equation}\label{ex:fluxfunction}
	Q(s):=\frac{s}{1+s^2} \qquad \mbox{ and } \qquad Q(s):=s e^{-s^2},
\end{equation}
which satisfy assumptions \eqref{eq:Q-ass2} with $\kappa = 1$ and $\kappa= \frac{1}{\sqrt{2}}$, respectively (see the left hand picture of Figure \ref{fig1}).
\begin{figure}[t]
	\begin{center}
		\includegraphics[width=7cm,height=5.7cm]{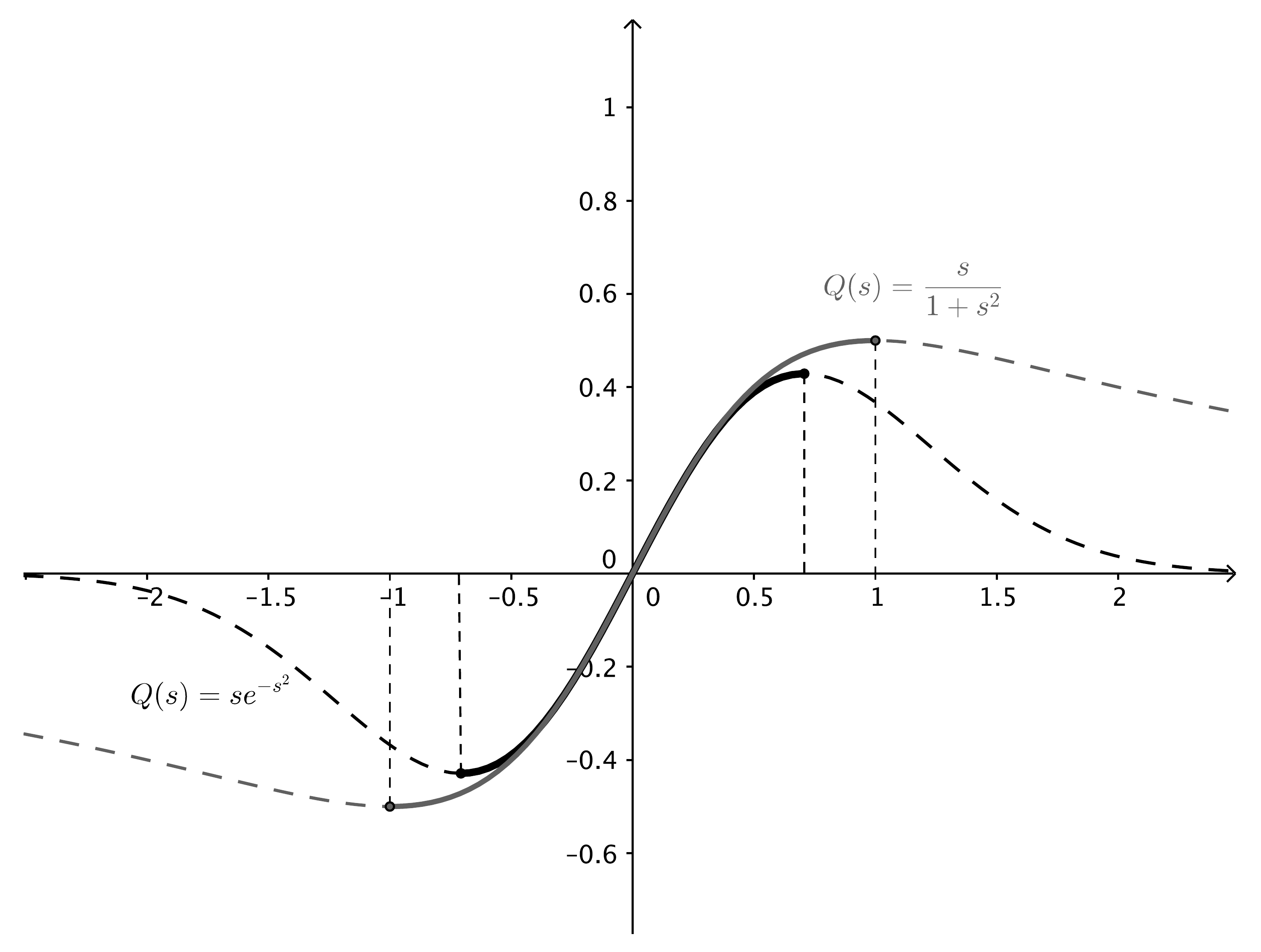}
		\,
		\includegraphics[width=7cm,height=5.7cm]{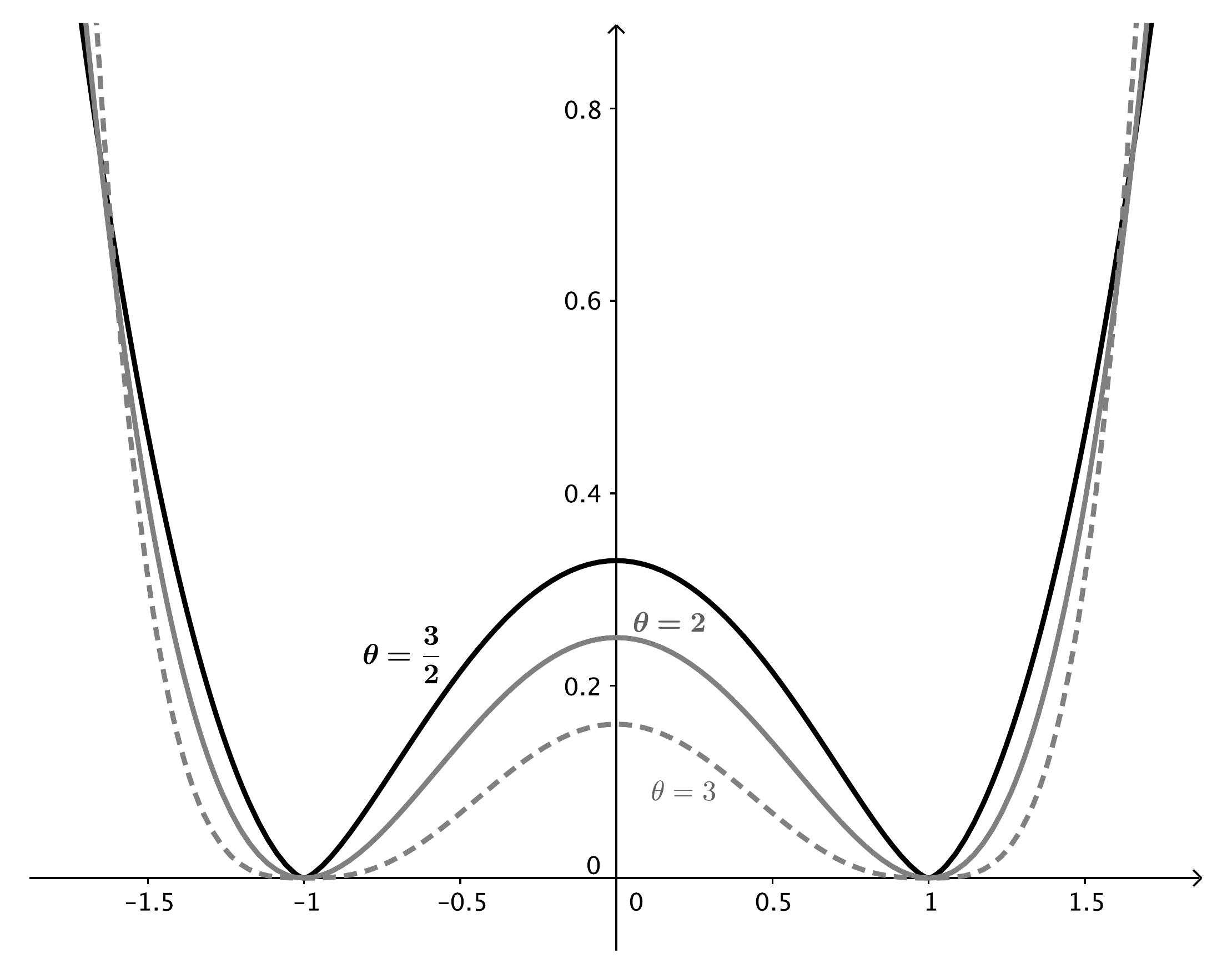}
		\hspace{3mm}
		\caption{\small{In the left hand picture we depict the two prototype examples for the diffusion $Q$; 
		in the right hand side there are several plots of the potential function \eqref{F:ex} for different choices of the parameter $\theta$}.}
		\label{fig1}
	\end{center}
\end{figure}
Regarding the reaction term, we require that the potential $F\in C^1(\R)$ satisfies
\begin{equation}\label{ipoF1}
	F(\pm1)=F'(\pm1)=0, \qquad \quad F(u)>0 \quad \forall\,u\neq\pm 1,
\end{equation}
and that there exist constants $0<\lambda_1\leq\lambda_2$, $\eta>0$ and $\theta>1$ such that:
\begin{equation}\label{ipoF2}
	\lambda_1 |1\pm u|^{\theta-2} \leq \frac{F'(u)}{u\pm1}\leq\lambda_2 |1\pm u|^{\theta-2}, \qquad\qquad \mbox{ for } \qquad |u\pm1|<\eta.
\end{equation}
Therefore, \eqref{ipoF1} ensures that $F$ is a double well potential with wells of equal depth in $u=\pm1$ and \eqref{ipoF2}
describes the behavior of $F$ close to the minimal points.
In particular, notice that by integrating \eqref{ipoF2} and using \eqref{ipoF1}, we obtain
\begin{equation}\label{eq:ass-F3}
		\frac{\lambda_1}{\theta}|1\pm u|^{\theta}\leq F(u)\leq \frac{\lambda_2}{\theta}|1\pm u|^{\theta}, 
		\quad \qquad |u\pm1|<\eta.
\end{equation}
The simplest example of potential $F$ satisfying \eqref{ipoF1}-\eqref{ipoF2} is
\begin{equation}\label{F:ex}
	F(u)=\frac{1}{2\theta}|1-u^2|^\theta, \quad \theta>1,
\end{equation}
which is depicted in Figure \ref{fig1} for different choices of $\theta>1$.
It is worth mentioning that when $\theta=2$ in \eqref{F:ex}, we obtain the classical double well potential $F(u)=\frac{1}{4}(1-u^2)^2$, 
which, in particular, satisfies $F''(\pm1)>0$; 
on the other hand, if $\theta\in(1,2)$ the second derivative of the potential \eqref{F:ex} blows up as $u\to\pm1$, 
while for $\theta>2$ we have the \emph{degenerate} case $F''(\pm1)=0$. 
Finally, notice that \eqref{ipoF1} implies that the reaction term $f=-F'$ satisfies $\displaystyle\int_{-1}^1 f(s)\,ds=0$, 
being the reason why we call $f$ a balanced bistable reaction term.

The competition between a balanced reaction term satisfying the additional assumption $F''(\pm1)>0$ and a classical linear diffusion is described by the celebrated Allen--Cahn equation \cite{Allen-Cahn}, which can be obtained from \eqref{eq:Q-model} by choosing $Q(s)=s$ and reads as
\begin{equation}\label{eq:A-C}
	u_t=\e^2u_{xx}-F'(u).
\end{equation}
Such model has been extensively studied since the early works \cite{Bron-Kohn,Carr-Pego,Fusco-Hale}, and
it is well known that the solution of \eqref{eq:A-C} subject to \eqref{eq:Neu}-\eqref{eq:initial} exhibits a peculiar phenomenon 
when the diffusion coefficient $\e>0$ is very small, known in literature as \emph{metastability}:
if the initial datum has a transition layer structure, that is $u_0$ is close to a step function taking values in $\{\pm1\}$ and has sharp transition layers, 
then the corresponding solution evolves very slowly in time and the layers move towards one another or towards the endpoints of the interval $(a,b)$ at an extremely low speed.
Once two layers are close enough or one of them is sufficiently close to $a$ or $b$, they disappear quickly and, after that, again the solution enters in a slow motion regime.
The latter phenomenon repeats until all the transitions disappear and the solution reaches a stable configuration, which is given by one of the two stable equilibria $u=\pm1$.
A rigorous description of such metastable dynamics first appeared in the seminal work \cite{Carr-Pego} where, in particular, it is proved that the time needed for the annihilation 
of the closest layers is of order $\exp(Al/\e)$, where $A:=\min\left\{F''(\pm1)\right\}>0$ and $l$ is the distance between the layers.
Therefore, the dynamics strongly depends on the parameter $\e>0$ and the evolution of the solution is extremely slow when $\e\to0^+$.
Moreover, the assumption $F''(\pm1)>0$ is necessary to have metastability and almost 25 years later than the publication of \cite{Bron-Kohn,Carr-Pego,Fusco-Hale},
in \cite{Bet-Sme} the authors prove that in the degenerate case $F''(\pm1)=0$ the exponentially small speed of the layers is replaced by an algebraic upper bound.
On the other hand, the slow motion phenomena described above appear only for potentials $F\in C^2(\R)$, while in the case of a potential of the form \eqref{F:ex} with $\theta\in(1,2)$
there exist stationary solutions to \eqref{eq:A-C}-\eqref{eq:Neu} that attain the values $\pm1$, with an arbitrary number of layers randomly located inside the interval $(a,b)$,
for details see \cite{Dra-Rob} or \cite{FPS-DCDS}.
In addition, in \cite{FPS-DCDS} a more general equation than \eqref{eq:A-C} is considered: 
the linear diffusion $u_{xx}$ is replaced by the (nonlinear) $p$-Laplace operator $(|u_x|^{p-2}u_x)_x$ and it is shown that the aforementioned phenomena strongly rely on the 
interplay between the parameters $p,\theta>1$.
To be more precise, there exist stationary solutions with a transition layer structure for any $\theta\in(1,p)$;
the metastable dynamics appears in the case $\theta=p$ and, finally, the solutions exhibit an algebraic slow motion for any $\theta>p>1$.    

The main novelty of our work consists in considering a general function $Q$ satisfying \eqref{eq:Q-ass1}-\eqref{eq:Q-ass2}, 
with the purpose of extending the aforementioned results to the reaction-diffusion model \eqref{eq:Q-model}. 
The choice of a function $Q$ satisfying \eqref{eq:Q-ass1}-\eqref{eq:Q-ass2} is inspired by \cite{PerMal}, 
where the authors introduced the so-called \emph{Perona--Malik equation} (PME)
\begin{equation}\label{eq:PME}
	u_t=Q(u_x)_x,
\end{equation}
with $Q$ as in \eqref{ex:fluxfunction}, to describe noise reduction and edge detection of digitalized images.
To be more precise, in \cite{PerMal} the authors consider the multi-dimensional version of \eqref{eq:PME} in the cylinder 
$\left\{(x,t)\in\R^3 \, : \, x\in\Omega,\, t\geq0\right\}$, where $\Omega\subset\R^2$ is a bounded open set, with initial datum $u(x,0)=u_0(x)$ and homogeneous Neumann boundary conditions on $\partial\Omega$.
The initial datum $u_0$ represents the brightness (or the grey level) of a picture which one wants to denoise and it is numerically shown that the evolution according to \eqref{eq:PME} 
smooths the zones where $|\nabla u_0|<1$ and enhance the zones with $|\nabla u_0|>1$, providing a sharper image than the initial one.
Ever since it was proposed by Perona and Malik in 1990, the nonlinear forward-backward heat equation \eqref{eq:PME} has attracted the interest of the mathematical community.
The main reason is that, on the one hand, numerical experiments exhibit good stability properties and produce the desired effect of fading out flat noise,
but on the other hand from the analytical point of view the forward-backward character of \eqref{eq:PME} induces a general skepticism, partially supported by some negative results about the \emph{ill-posedness} nature of \eqref{eq:PME}. 
This is usually referred in literature as ``the Perona--Malik paradox'' \cite{Kich}.
Without claiming to be complete, we list some fundamental analytical results about \eqref{eq:PME}.
In \cite{Gobbi}, the author prove that if $u:\R^2\to\R$ is a $C^1$ solution of \eqref{eq:PME}, then there exist $a,b\in\R$ such that $u(x,t)=ax+b$, for any $(x,t)\in\R^2$.
It is important to mention that such a result is a consequence of the forward-backward character of the equation, since it is well known that it is false, in general, for forward parabolic equation: for instance, the function $u(x,t)=e^t\sin(x)$ is an entire solution of the classical heat equation.
The initial boundary value problem (IBVP) associated to \eqref{eq:PME} with homogeneous Neumann boundary conditions \eqref{eq:Neu} has been investigated in \cite{Kaw-Kut}, where the authors prove the following results:
\begin{itemize}
	\item If the initial datum satisfies $|u_0'(x)|>\kappa$, for some $x\in[a,b]$, then (IBVP) does not admit any $C^1$ solution defined for every $t\geq0$.
	\item $C^1$ solutions of (IBVP) are unique. This is false if one considers \emph{weaker} solutions as it is shown in \cite{Hollig}.
	\item If the initial datum satisfies $|u_0'(x)|<\kappa$, for any $x\in[a,b]$, then there exists a unique global classical solution, namely a unique solution $u\in C^2((a,b)\times(0,T))$, 
		for any $T>0$.
\end{itemize}
Finally, we recall that in \cite{Guido} it is shown that PME admits a natural regularization 
by forward-backward diffusions possessing better analytical properties than PME itself. 

Many papers have also been devoted to the study of \eqref{eq:PME} with the presence of reaction and/or convection terms:
for instance, we mention \cite{FS-NA} where a Burger's type equation with Perona--Malik diffusion is considered
and \cite{CorMalSov}, where the authors study existence of wavefront solutions for a reaction-convection equation with Perona--Malik diffusion.
Regarding the reaction-diffusion model \eqref{eq:Q-model}, we recall that the corresponding multi-dimensional model is studied in \cite{Morfu}, 
where it is shown that combining the properties of an anisotropic diffusion like the Perona--Malik's
with those of bistable reaction terms provides a better processing tool which enables noise filtering, contrast enhancement and edge preserving.

Up to our knowledge, the long time behavior of phase transition layer solutions to reaction-diffusion models with a nonlinear diffusion of the form $Q(u_x)_x$
has been studied only in the case of the $p$-Laplacian \cite{FPS-DCDS} and mean curvature operators \cite{FPS,FS}.
Both cases are rather different than \eqref{eq:Q-ass1}-\eqref{eq:Q-ass2}:
to highlight the differences, let us expand the term $Q(u_x)_x$ as $Q'(u_x)u_{xx}$ and notice that the derivative of $Q$ plays the role of the diffusion coefficient.
In the case of the $p$-Laplacian, $Q'$ is singular at $s=0$ for $p\in(1,2)$ and degenerate ($Q'(0)=0$) for $p>2$,
while in the cases studied in \cite{FPS,FS} the function $Q$ is explicitly given by
$$Q(s)=\frac{s}{\sqrt{1+s^2}}, \qquad \mbox{ or } \qquad Q(s)=\frac{s}{\sqrt{1-s^2}}.$$
Hence, in the first case the derivative $Q'$ is strictly positive for any $s\in\R$ and since $Q'(s)\to0$ as $|s|\to\infty$, one has a degenerate diffusion coefficient
for large values of $|u_x|$, while in the latter case $Q'(s)\to\pm\infty$ as $s\to\pm1$ and, as a consequence, the diffusion coefficient is strictly positive but singular at $\pm1$.
In the case considered in this article, \eqref{eq:Q-ass1}-\eqref{eq:Q-ass2} imply that not only the diffusion coefficient is degenerate at $\pm\infty$ and $\pm\kappa$, 
but the diffusion coefficient is even strictly negative when $|u_x|>\kappa$. 
Thus, our work provides the first investigation of long time behavior of phase transition layer solutions in the case of a degenerate and negative diffusion.
Moreover, we mention that in \cite{FPS,FPS-DCDS,FS} an explicit formula for $Q$ is considered, 
while here $Q$ is not explicit but it is a generic function satisfying \eqref{eq:Q-ass1}-\eqref{eq:Q-ass2}.
Actually, the interested reader can check that even if considering  specific examples as in \eqref{ex:fluxfunction}, one cannot obtain explicit formulas for the  involved functions,  so that
the computations become much more complicated; for details, see Remark \ref{rem:difficile}. 

The main contribution of this paper is to show that the well know results about the classic model \eqref{eq:A-C} described above can be extended to \eqref{eq:Q-model},
even if $Q$ is a non-monotone function;
roughly speaking, the condition $Q'>0$ in $(-\kappa,\kappa)$ for some generic $\kappa$ is enough to obtain existence of \emph{compactons},
which are stationary solutions (hence, invariant under the dynamics of \eqref{eq:Q-model}), with a transition layer structure, 
in the case $F$ satisfy \eqref{ipoF1}-\eqref{ipoF2} with $\theta\in(1,2)$, 
existence of metastable patterns if $\theta=2$ and existence of algebraically slowly moving structures when $\theta>2$. 
The dynamical (in)stability of the compactons remains an interesting open problem;
it is not clear whether small perturbations of compactons generate a slow motion dynamics similar to the case $\theta=2$ 
or the transition layer structure is maintained for all times $t>0$.

\subsection*{Plan of the paper} We close the Introduction with a short plan of the paper; Section \ref{stationary} is devoted to the stationary problem associated to \eqref{eq:Q-model}. We will consider steady states both in the whole real line and in bounded intervals, showing also that there are substantial differences depending of the value of the power $\theta$ appearing in \eqref{ipoF2}. 
Indeed, the existence of the aforementioned compactons is a peculiarity of the case $\theta\in(1,2)$ and it is established in Proposition \ref{prop:comp}.
On the contrary, in the case $\theta\geq2$ we focus our attention on the existence of periodic solutions in the real line and their restriction on a bounded interval $[a,b]$ 
(see Propositions \eqref{periodicR} and \ref{periodic:bounded} respectively) that oscillate among values $\pm\bar s$, with $\bar s\approx1$ (strictly less that one). In Section \ref{sec:energy} we prove some variational results and lower bounds on the energy associated to \eqref{eq:Q-model} (for its definition, we refer to \eqref{eq:energy}) which will be crucial in order to prove the results of Section \ref{sec:slow}; we here focus on the asymptotic behavior of the solutions to \eqref{eq:Q-model}-\eqref{eq:Neu}-\eqref{eq:initial}, showing that if the dynamics starts from an initial datum with $N$ transition layers inside the interval $[a,b]$, then such configuration will be maintained for extremely long times; 
as it was previously mentioned, the time taken for the solution to annihilate the \emph{unstable} structure of the initial datum strongly depends on the choice of the parameter $\theta>1$ appearing in \eqref{ipoF2} and we have either metastable dynamics (exponentially slow motion) in the critical case $\theta=2$ or algebraic slow motion in the degenerate case $\theta>2$. We underline again that in the case $\theta \in (1,2)$  solutions with  $N$ transition layers are either stationary solutions (compactons) or close to them.
Numerical simulations, which illustrate the analytical results, are  provided at the end of Section \ref{sec:slow}.

\section{Stationary solutions}\label{stationary}
The aim of this section is to analyze the stationary problem associated to \eqref{eq:Q-model} and to prove the existence of some \emph{special} solutions to
\begin{equation}\label{eq:staz}
	Q(\e^2\varphi')'-F'(\varphi)=0,
\end{equation}
with $\e>0$, $Q$ satisfying assumptions \eqref{eq:Q-ass1}-\eqref{eq:Q-ass2} and $F$ as in \eqref{ipoF1}-\eqref{ipoF2}, 
both in the whole real line and in a bounded interval complemented with homogeneous Neumann boundary conditions \eqref{eq:Neu}. 
In order to prove the results of this section we actually have to require more regularity on the diffusion flux $Q$, that is $Q \in C^3(\R)$; however, we underline that the basic examples \eqref{ex:fluxfunction} we have in mind satisfy such additional assumption as well.

\subsection{Standing waves} We start by considering problem \eqref{eq:staz} in $\R$, and we focus the attention on standing waves, that can be defined as follows:
an increasing standing wave $\Phi_\e:=\Phi_\e(x)$ is a solution to \eqref{eq:staz} in the whole real line satisfying either
\begin{equation}\label{eq:stand-teta<2}
	  \begin{cases}
	  	\Phi_\e(x)=1, \qquad & x\geq x_+,\\
	  	\Phi_\e(x)=-1, & x\leq x_-, 
	  \end{cases} \qquad \qquad 
	  \Phi'_\e(x)> 0, \qquad \mbox{ for any } x\in(x_-,x_+),
\end{equation}
for some $x_\pm \in \R$ with $x_-<x_+$ or
\begin{equation}\label{eq:stand-teta>2}
	\lim_{x\to\pm\infty}\Phi_\e(x)=\pm1, \qquad \qquad  \Phi'_\e(x)> 0, \qquad \mbox{ for any } x\in\R.
\end{equation}
Similarly, a decreasing standing wave $\Psi_\e:=\Psi_\e(x)$ satisfies \eqref{eq:staz} and either
\begin{equation*}
	\begin{cases}	
		\Psi_\e(x)=-1, \qquad & x\geq x_+,\\
		\Psi_\e(x)=1, & x\leq x_-, 
	\end{cases} \qquad \qquad 
	\Psi'_\e(x)<0, \qquad \mbox{ for any } x\in(x_-,x_+),
\end{equation*}
for some $x_\pm \in \R$ with  $x_-<x_+$ or
\begin{equation*}
	\lim_{x\to\pm\infty}\Psi_\e(x)=\mp1, \qquad \qquad  \Psi'_\e(x)<0, \qquad \mbox{ for any } x\in\R.
\end{equation*}
It is easy to check that solutions to \eqref{eq:staz}-\eqref{eq:stand-teta<2} and \eqref{eq:staz}-\eqref{eq:stand-teta>2} are invariant by translation; thus, in order to deduce a unique solution we add the further assumption $\Phi_\e(0)=0$ and we  rewrite the problems \eqref{eq:staz}-\eqref{eq:stand-teta<2} and \eqref{eq:staz}-\eqref{eq:stand-teta>2} as 
\begin{equation}\label{eq:Fi}
	\e^2Q'(\e^2\Phi'_\e)\Phi''_\e-F'(\Phi_\e)=0, \qquad \lim_{x\to\pm\infty}\Phi_\e(x)=\pm1, \qquad \Phi_\e(0)=0, \qquad  \Phi'_\e(x)\geq 0,
\end{equation}
for any $x\in\R$. 
Analogously, in the decreasing case we have
\begin{equation}\label{eq:Fi-dec}
	\e^2Q'(\e^2\Psi'_\e)\Psi''_\e-F'(\Psi_\e)=0, \qquad \lim_{x\to\pm\infty}\Psi_\e(x)=\mp1, \qquad \Psi_\e(0)=0, \qquad  \Psi'_\e(x)\leq 0,
\end{equation}
for any $x\in\R$.

As we will see below, there is a fundamental difference whether $F$ satisfies \eqref{ipoF2} with $\theta\in(1,2)$ or with $\theta>2$: in the first case, the standing waves touch the values $\pm1$, namely the increasing standing waves satisfy \eqref{eq:stand-teta<2}.
Conversely, if $\theta>2$ the standing waves reach the values $\pm1$ only in the limit: for instance, in the increasing case, they satisfy \eqref{eq:stand-teta>2}. In order to prove such claim, as well as the existence of a unique solution to \eqref{eq:Fi} (or, alternatively, of \eqref{eq:Fi-dec}), we need to premise the following technical result.

\begin{lem}\label{lemma:new}
Let $Q \in C^1(\R)$ satisfying \eqref{eq:Q-ass1}-\eqref{eq:Q-ass2}. 
Denote by
\begin{equation}\label{def:ell}
	\ell := \kappa Q(\kappa)-\tilde Q(\kappa), \qquad \mbox{ where } \qquad \tilde Q(s) := \int_0^s Q(t) \, dt,
\end{equation}
and
\begin{equation}\label{def:P}
\begin{aligned}
	P_\e(s):= \int_{0}^{s}\e^2z\,Q'(\e^2 z)\,dz.
	\end{aligned}
\end{equation}
Then, there exists a unique (strictly positive) function $J_\e$ which inverts the equation $P_\e(s)=\xi$ in $[0,\kappa \e^{-2}]$:
for any $s\in[0,\kappa \e^{-2}]$ and $\xi\in[0,\ell\e^{-2}]$ there holds
\begin{equation}\label{eq:J_e}
	P_\e(s)=\xi \qquad \qquad \mbox{ if and only if } \qquad \qquad s=J_\e(\xi).
\end{equation}
Moreover, the following expansion holds true
\begin{equation}\label{eq:J-important}
	J_\e(\xi) = \sqrt{\frac{2}{\e^2Q'(0)}\xi}+\e^{-2}\rho(\e^2\xi),
	\qquad \mbox{ where } \qquad \rho(\xi)=o(\xi).
\end{equation}
\end{lem}
\begin{proof}
In order to study the invertibility of the equation $P_\e(s)=\xi$, we observe that
\begin{equation}\label{defalternativa}
\begin{aligned}
	P_\e(s)& =\int_{0}^{s}\e^2z\,Q'(\e^2 z)\,dz =\frac{1}{\e^2} \int_0^{\e^2  s } \tau  Q'(\tau) \, d\tau  \\
	&= \frac{1}{\e^2} \left[ \tau \,  Q(\tau) \Big|^{\e^2 s}_0- \int_0^{\e^2 s} Q(\tau) \, dt\right] = s Q(\e^2 s) - \frac{1}{\e^2}\tilde Q (\e^2 s).
	\end{aligned}
\end{equation}
Hence $P_\e$ is an even function satisfying 
$P_\e'(0)=P_\e'(\pm\kappa\e^{-2})=0$ and
\begin{equation*}
	P_\e'(s)s= \e^2 Q'(\e^2 s) \, s^2 \geq 0, \qquad \mbox{ for any } \; s \in[-\kappa \e^{-2},\kappa \e^{-2}],
\end{equation*}
because of the definition \eqref{def:P} and the assumptions on $Q$ \eqref{eq:Q-ass1}-\eqref{eq:Q-ass2}.
Thus, the equation $P_\e(s)=\xi$ has exactly two solutions in $[-\kappa \e^{-2},\kappa \e^{-2}]$, 
provided that $\xi\in[P_\e(0),P_\e(\kappa \e^{-2})]$.
Going further, one has $P_\e(0)=0$ while
\begin{equation*}
	P_\e(\kappa \e^{-2})= \kappa \e^{-2} Q(\kappa)-\e^{-2} \tilde Q(\kappa) = \frac{\ell}{\e^2},
\end{equation*}
where the constant $\ell$, defined in \eqref{def:ell}, is strictly positive because $\kappa \, Q(\kappa)$ is indeed greater than $\tilde Q(\kappa)$, 
which represents the area underneath the function $Q$ in the interval $[0,\kappa]$. 
Hence, \eqref{eq:J_e} holds true and it remains to prove \eqref{eq:J-important}.
Using \eqref{defalternativa}, we deduce 
\begin{equation}\label{eq:defH}
	P_\e(s)=\e^{-2}H(\e^2s), \qquad \mbox{ where} \qquad H(s):=sQ(s)-\tilde Q(s).
\end{equation}
Since the function $H$ in \eqref{eq:defH} satisfies $H(0)=H'(0)=0$ and $H''(0)=Q'(0)>0$, for $s \sim 0$ we have 
$$P_\e(s) =\e^{-2}\left[\frac{Q'(0)}2 \,\e^4s^2 +o((\e^2s)^2)\right],$$ 
and the latter equality gives a hint that $J_\e$ behaves like the square root of $s$ for $s \sim 0$.
To prove it, let us study the behavior of  $(H^{-1})^2$ close to the origin. 
We have
\begin{equation*}
	\left[(H^{-1})^2 \right]'(s)= 2 H^{-1}(s) \left[H^{-1}\right]' (s)= \frac{2H^{-1}(s)}{H'\left(H^{-1}(s)\right)},
\end{equation*}
so that
\begin{align*}
	\left[(H^{-1})^2 \right]'(0) &= 2\lim_{s \to 0}  \frac{H^{-1}(s)}{H'\left(H^{-1}(s)\right)} = 2 \lim_{s \to 0}  \frac{H^{-1}(s)}{H''(0)H^{-1}(s) + o\left(H^{-1}(s)^2\right)} \\
	&= 2\lim_{s \to 0} \frac{1}{H''(0)+ \frac{o\left(H^{-1}(s)^2\right)}{H^{-1}(s)}} = \frac{2}{H''(0)} =\frac{2}{Q'(0)}>0.
\end{align*}
Thus,
\begin{equation*}
	(H^{-1})^2(s) = 0 +\frac{2}{Q'(0)} s + R(s) \quad \Longrightarrow \quad H^{-1}(s) = \sqrt{\frac{2}{Q'(0)} s + R(s)},
\end{equation*}
where  $R(s)=\mathcal{O}(s^2)$. 
Hence, we can state that
\begin{equation*}
	H^{-1}(s) = \sqrt{\frac{2}{Q'(0)} s} + \rho(s), \qquad \mbox{ with } \qquad \rho(s)= o(s). 
\end{equation*}
Indeed,
\begin{align*}
	\lim_{s \to 0^+}\frac{\rho(s)}{s}&=\lim_{s \to 0^+}s^{-1}\left[\sqrt{\frac{2}{Q''(0)}s+R(s)}-\sqrt{\frac{2}{Q''(0)}s} \, \right]=
	\lim_{s \to 0^+}\frac{s^{-1}R(s)}{\sqrt{\frac{2}{Q''(0)}s+R(s)}+\sqrt{\frac{2}{Q''(0)}s}} \\
	&=\lim_{s \to 0^+}\frac{s^{-\frac{3}{2}}R(s)}{\sqrt{\frac{2}{Q''(0)}+s^{-1}{R(s)}}+\sqrt{\frac{2}{Q''(0)}}} =0.
\end{align*}
By using \eqref{eq:J_e}-\eqref{eq:defH}, one obtains
\begin{equation*}
	J_\e(\xi) =\e^{-2}H^{-1}(\e^2\xi)=\sqrt{\frac{2}{\e^2Q'(0)}\xi} +\e^{-2}\rho(\e^2\xi), 
\end{equation*}
that is \eqref{eq:J-important} and the proof is complete.
\end{proof}
\begin{rem}\label{rem:difficile}
It is interesting to notice that, even if we consider the explicit examples in \eqref{ex:fluxfunction}, 
it is not possible to give an explicit formula for the function $J_\e$ in \eqref{eq:J_e}.
As we will see in the rest of the paper, not having an explicit formula for $J_\e$ considerably complicates the proof of our results, in which the expansion \eqref{eq:J-important} plays a crucial role.
\end{rem}

We have now all the tools to prove the following existence result.

\begin{prop}\label{ESW}
Let $Q$ satisfying \eqref{eq:Q-ass1}-\eqref{eq:Q-ass2} and $F$ satisfying \eqref{ipoF1}-\eqref{ipoF2}.
Then, there exists $\e_0>0$ such that problem  \eqref{eq:Fi} admits a unique  solution $\Phi_\e \in C^2(\R)$  for any $\e \in (0,\e_0)$.
Moreover we have the following alternatives:
\begin{itemize}
\item[(i)] If $\theta\in(1,2)$, then the profile $\Phi_\e$ satisfies \eqref{eq:stand-teta<2}; more precisely one has
\begin{equation}\label{ugualea1e-1}
	\Phi_\e(x^\e_1)=1, \qquad \qquad  \Phi_\e(x^\e_2)=-1, 
\end{equation}
where 
\begin{equation}\label{eq:x-eps}
	x^\e_1=\e \bar x_1+o(\e), \qquad \qquad x^\e_2=-\e \bar x_2+o(\e),
\end{equation}
for some $\bar x_i>0$, for $i=1,2$ which do not depend on $\e$.
\item[(ii)] If $\theta=2$, then  $\Phi_\e$ satisfies \eqref{eq:stand-teta>2} with the following {\it exponential decay}:
\begin{equation*}
	\begin{aligned}
		&1-\Phi_\e(x)\leq c_1 e^{-c_2x}, \qquad &\mbox{as } x\to+\infty,\\
		&\Phi_\e(x)+1\leq c_1 e^{c_2x}, &\mbox{as } x\to-\infty,
	\end{aligned}
\end{equation*}
for some $c_1,c_2>0$.
\item[(iii)] If $\theta>2$, then $\Phi_\e$ satisfies \eqref{eq:stand-teta>2} with {\it algebraic decay}:
\begin{equation*}
	\begin{aligned}
		&1-\Phi_\e(x)\leq d_1 x^{-d_2}, \qquad &\mbox{as } x\to+\infty,\\
		&\Phi_\e(x)+1\leq d_1x^{-d_2}, &\mbox{as } x\to-\infty,
	\end{aligned}
\end{equation*}
for some $d_1,d_2>0$.
\end{itemize}
\end{prop}

 \begin{proof}
In order to prove the existence of a  solution to  \eqref{eq:Fi}, we multiply the ordinary differential equation
by $\Phi'_\e=\Phi'_\e(x)$,  deducing
\begin{equation*}
	\e^2Q'(\e^2\Phi'_\e)\Phi'_\e\Phi''_\e-F'(\Phi_\e)\Phi'_\e=0,  \qquad\qquad \mbox{in }\, (-\infty,+\infty).
\end{equation*}
As a consequence,
\begin{equation}\label{eq:SW}
	\begin{cases}
		P_\e(\Phi'_\e)=F(\Phi_\e),  \qquad \qquad \mbox{in }\, (-\infty,+\infty),\\
		\Phi_\e(0)=0,
	\end{cases} 
\end{equation}
where $P_\e$ is defined in \eqref{def:P}.
In order to solve the Cauchy problem \eqref{eq:SW}, we apply Lemma \ref{lemma:new};  
hence, we need to require $F(\Phi_\e) \leq \ell \e^{-2}$, namely we choose 
\begin{equation}\label{eq:maxF}
	\e\in(0,\e_0),  \qquad \mbox{ with } \qquad \e_0=\sqrt{\frac{\ell}{\displaystyle\max_{\Phi\in[-1,1]}F(\Phi)}}.
\end{equation}
Condition \eqref{eq:maxF} ensures that we can find monotone solutions to \eqref{eq:SW} 
applying the standard method of separation of variables;
in particular, we obtain a solution $\Phi_\e$, satisfying $\Phi'_\e \in [0,\kappa \e^{-2}]$, 
which is implicitly defined by 
 \begin{equation*}
	 \int_0^{\Phi_\e(x)} \frac{du}{J_\e\left(F(u)\right) }= x,
 \end{equation*}
where $J_\e$ is defined in \eqref{eq:J_e}.
Since $J_\e\left(F(u)\right)=0$ if and only if $F(u)=0$ (that is, $u= \pm 1$), in order to prove the uniqueness of $\Phi_\e$ 
and its behavior described in the properties (i), (ii) and (iii) of the statement, 
we need to study the convergence of the following improper integrals
\begin{equation}\label{int_SW}
	\int_{0}^1 \frac{du}{J_\e\left( F(u)\right)} \qquad \mbox{ and } \qquad \int_{-1}^0 \frac{du}{J_\e\left( F(u)\right)}.
\end{equation}
Substituting \eqref{eq:J-important} in the first integral of \eqref{int_SW}, we end up with 
\begin{equation}\label{int_SW2}
	\int_{0}^1 \frac{du}{J_\e\left( F(u)\right)}=\e\bar x_1+I(\e),
\end{equation}
where 
\begin{align*}
	\bar x_1&:=\sqrt{\frac{Q'(0)}{2}}\int_0^1 \frac{du}{\sqrt{F(u)}}\\
	I(\e)&:=-\e^{-1}\sqrt{\frac{Q'(0)}{2}}\int_0^1\frac{\rho(\e^2F(u))}{\left(\sqrt{\frac{2}{\e^2Q'(0)} F(u)} +\e^{-2}\rho(\e^2F(u))\right)\sqrt{F(u)}}\,du.
\end{align*}
The crucial point is that the character of the integral in \eqref{int_SW2} is simply given by $\bar x_1$, 
since there exists $\e_0>0$ such that $|I(\e)|<\infty$ for any $\e\in(0,\e_0)$. 
Indeed, using the estimate $|\rho(\e^2F(u))| \leq C \e^2 F(u)$, one gets
\begin{align*}
	\int_0^1 \frac{|\rho(\e^2F(u))|}{\left(\sqrt{\frac{2}{\e^2Q'(0)} F(u)} +\e^{-2}\rho(\e^2F(u))\right)\sqrt{F(u)}}\,du &\leq 
	\int_0^1  \frac{C \e^2\sqrt{F(u)}}{\left(\sqrt{\frac{2}{\e^2Q'(0)} F(u)} +\e^{-2}\rho(\e^2F(u))\right)}\,du \\
	& \leq  \int_0^1  \frac{C \e^3}{\left(\sqrt{\frac{2}{Q'(0)} }+\frac{\rho(\e^2F(u))}{\e\sqrt{F(u)}}\right)}\,du,
\end{align*}
and we can choose $\e>0$ sufficiently small such that
\begin{equation*}
	\sqrt{\frac{2}{Q'(0)} } +\frac{\rho(\e^2F(u))}{\e\sqrt{F(u)}} \geq \sqrt{\frac{2}{Q'(0)} } -C\e\sqrt{F(u)} >0.
\end{equation*}
Moreover, notice that $\bar x_1$ does not depend on $\e$, while $I(\e)=o(\e)$.
By using \eqref{eq:ass-F3}, we obtain
\begin{equation*}
	\bar x_1 \sim  \int_0^1 \frac{du}{(1-u)^{\frac{\theta}{2}}}.
\end{equation*}
Hence, $\bar x_1 < +\infty$ if and only if $\theta< 2$, and the point (i) of the thesis follows;
the first equality in \eqref{eq:x-eps} is a consequence of \eqref{int_SW2}.
Going further, points (ii)-(iii) of the statement are a consequence of the standard theory of ODE applied to \eqref{eq:SW}, together with the fact that 
$$J_\e(F(s)) \approx |1-s|^{\frac{\theta}{2}}, \qquad s \approx 1.$$
The  computations are completely similar if considering the second integral in \eqref{int_SW}, and we thus proved the existence of  a unique solution of \eqref{eq:Fi} satisfying properties (i), (ii) and (iii).
Precisely,  if $\theta\geq2$ there exists a unique solution of \eqref{eq:SW}, while if $\theta\in(1,2)$, \eqref{eq:SW} has infinitely many solutions, but
the additional requirement $\Phi_\e'(x)\geq0$, for any $x\in\R$, guarantees that there is a unique solution of \eqref{eq:Fi}.
\end{proof}

\begin{rem}We notice that the condition $\Phi'_\e \in [0,\kappa \e^{-2}]$ allows also for high values of the first derivative; to be more precise, one has 
\begin{equation}\label{eq:Fi-firstderivative}
	|\Phi_\e'(x)|\leq \kappa\e^{-2}, \qquad \qquad \mbox{ for any } x\in\R.
\end{equation}
\end{rem}

As a corollary of Proposition \ref{ESW}, we can prove existence of a unique solution to \eqref{eq:Fi-dec}, sharing similar properties to (i)-(ii) and (iii). 
\begin{cor}\label{cor:standwavedecr}
Under the same assumptions of Proposition \ref{ESW}, there exists $\e_0>0$ such that problem \eqref{eq:Fi-dec} admits a unique solution 
$\Psi_\e$  for any $\e \in (0,\e_0)$. 
Moreover, if $\theta\in(1,2)$, then
\begin{equation*}
	\Psi_\e(-x^\e_1)=1, \qquad \qquad  \Psi_\e(-x^\e_2)=-1, 
\end{equation*}
where $x^\e_i$, $i=1,2$ are defined in \eqref{eq:x-eps}.
On the other hand, if $\theta=2$ ($\theta>2$) the profile $\Psi_\e$ has an exponential (algebraic) decay towards the states $\mp1$.
\end{cor}
\begin{proof}
Using the symmetry of $Q$ and, in particular, the fact that $Q'$ is an even function, it is a simple exercise to verify that $\Psi_\e(x):=\Phi_\e(-x)$, 
with $\Phi_\e$ given by Proposition \ref{ESW}, is the unique solution to \eqref{eq:Fi-dec}.
\end{proof}

\begin{rem}
The existence of a unique  solution $\Psi_\e$ to  \eqref{eq:Fi-dec} can be proven independently on the one of $\Phi_\e$: 
indeed, it is enough to adapt the proof of Proposition \ref{ESW} by inverting the equation $P_\e(s)=\xi$ in the interval $[-\kappa\e^{-2},0]$ 
(in this case the inverse is $-J_\e$, see \eqref{eq:J_e}), 
and obtaining the existence of a unique solution $\Psi_\e$ with negative derivative $\Psi'_\e \in [-\kappa \e^{-2},0]$.
\end{rem}

The previous results are instrumental to prove the existence of a {\it special} class of stationary solutions on a bounded interval in the case $\theta \in (1,2)$, 
as we will see in the next section. 
Indeed, in such a case the standing waves reaches $\pm 1$ for a finite value of the $x$-variable and with zero derivative, 
so that we are able to construct infinitely many steady states oscillating between $\pm 1$ and satisfying the boundary conditions \eqref{eq:Neu}. 
On the contrary, if $\theta \geq 2$, all the standing waves satisfy \eqref{eq:stand-teta>2}, so that they never satisfy
the homogeneous Neumann boundary conditions and can never solve \eqref{eq:staz}-\eqref{eq:Neu} in any bounded interval.

\subsection{Compactons}
We here consider the so-called \emph{compactons}, which are by definition stationary solutions connecting two phases on a finite interval. 
More explicitly, we prove the existence of infinite  solutions to the stationary problem associated to \eqref{eq:Q-model}-\eqref{eq:Neu}, namely
\begin{equation}\label{eq:comp}
	Q(\e^2 \varphi')'-F'(\varphi)=0, \qquad \qquad \varphi'(a)=\varphi'(b)=0,
\end{equation}  
oscillating between $-1$ and $+1$ (touching them), provided that $\e$ is sufficiently small. 
The existence of such solutions is shown by proving that, given an arbitrary set of real numbers in $[a,b]$, for sufficiently small $\e$, 
there are two solutions $\varphi_1$ and $\varphi_2$ to \eqref{eq:comp} having such numbers as zeros, satisfying 
\begin{equation}\label{varphi12}
	\varphi_1(a)=-1 \qquad\quad  \mbox{ and } \quad \qquad \varphi_2(a)=+1,
\end{equation}
and oscillating between $-1$ and $+1$ ( $+1$ and $-1$, respectively).   

\begin{prop}\label{prop:comp}
Let $1<\theta<2$, $N\in \mathbb{N}$ and let $h_1,h_2,\dots, h_N$ be any $N$ real numbers such that $a<h_1<h_2<\cdots <h_N<b$. 
Then, for any $\e\in(0,\bar\e)$ with $\bar\varepsilon>0$ sufficiently small, there exist two solutions $\varphi_1$ and $\varphi_2$ to \eqref{eq:comp} satisfying \eqref{varphi12}, oscillating between $-1$ and $+1$ and between $+1$ and $-1$ respectively, and having precisely $N$ zeros at $h_1,h_2,\dots, h_N$. 
\end{prop}

\begin{proof}
We start by proving the existence of the solution $\varphi_1$ to \eqref{eq:comp} on the interval $[a,b]$ satisfying the first condition in  \eqref{varphi12}, oscillating between $-1$ and $+1$ and having $h_1,h_2,\dots, h_N$ as zeros.
To this aim, we consider the function 
$$\Phi_{\e}^1(x):=\Phi_\e(x-h_1), \qquad \,  x\in\R,$$ 
where $\Phi_\e$ is the increasing standing wave solution of Proposition \ref{ESW}. 
Then, the function $\Phi_{\e}^1$ has a zero at $h_1$.
Furthermore, by \eqref{eq:Fi} and \eqref{ugualea1e-1}, recalling that 
\begin{equation*}
	x_1^\e =\e \bar{x}_1+o(\e) \quad \text{and}\quad x_2^\e =-\e \bar{x}_2+o(\e),
\end{equation*} 
for some $\bar{x}_1, \bar{x}_2>0$,
we can conclude that $\Phi_{\e}^1$ takes the values $-1$ on $(-\infty, h_1+x_2^\e]$ and $+1$ on $[h_1+x_1^\e, +\infty)$;
we notice that if $\e$ is sufficiently small, then $h_1+x_2^\e  <h_1$. 
Let us now fix $\e>0$ small enough so that $h_1+x_2^\e >a$; 
the restriction of $\Phi_{\e}^1$ to the interval $[a,h_1+x_1^\e]$, denoted with the same symbol, turns out to be equal to $-1 $ for every $x\in [a,h_1+x_2^\e]$, touching $+1$ at $h_1+x_1^\e$. 
Let us introduce the notation 
\begin{equation*}
y_j^\e:=
\begin{cases}
x_1^\e \qquad & \text{if $j$ is odd}, \\
-x_2^\e  &  \text{if $j$ is even}, \\
\end{cases}	
\end{equation*}
for every $j=1,\dots, N$, and for every $i=1,\dots, N-1$ let us  define the function  
\begin{equation*}
	\Phi_{\e}^{i+1}(x):=\Phi_\e((-1)^{i}(x-h_{i+1})),  \qquad x\in \mathbb{R}.
\end{equation*}
For each fixed $i=1,\dots, N-1$, the function $\Phi_{\e}^{i+1}$ has a zero at $h_{i+1}$, and takes the values $-1$ for every $x \geq h_{i+1}+y_{i+1}^\e$ and $+1$ for every $x\leq h_{i+1}-y_{i}^\e$ as long as $i$ is odd, otherwise it takes the values $-1$ for every $x \leq h_{i+1}-y_{i}^\e$ and $+1$ for every $x\geq h_{i+1}+y_{i+1}^\e$. 
Up to choosing $\e$ possibly smaller in order to have $h_i+ y_i^\e \leq h_{i+1} -y_i^\e$, for every $i=1,\dots, N-1$ (namely $2y_i^\e \leq h_{i+1} -h_i $), 
the restriction of the function $\Phi_{\e}^{i+1}(x)$ to the interval $[h_i+y_i^\e, h_{i+1}+y_{i+1}^\e ]$ (still using the same notation), takes the value $(-1)^{i+1}$ for every $x\in [h_i+y_i^\e, h_{i+1}-y_i^\e]$, touching $(-1)^i$ at $h_{i+1}+y_{i+1}^\e$.  
Selecting $\e>0$ sufficiently small so that $h_N+y_N^\e< b$, we end up defining $\varphi_1$ in the following way
\begin{equation*}
\varphi_1(x):=
\begin{cases}
\Phi_\e^1(x), \qquad & x\in [a,h_1+x_1^\e], \\
\Phi_\e^{i+1}(x),  & x\in [h_i+y_i^\e, h_{i+1}+y_{i+1}^\e]\quad  \text{ and }i=1,\dots,N-1, \\
(-1)^{N-1}, & x\in [h_N+y_N^\e,b].
\end{cases}	
\end{equation*}
{By construction}, the resulting map {$\varphi_1\in C^2([a,b])$} solves \eqref{eq:comp} and has exactly $N$ zeros at points $h_1,h_2,\dots, h_N$.

Arguing as above, we can construct the compacton $\varphi_2$ which satisfies the second condition in  \eqref{varphi12} in the following way
 \begin{equation*}
\varphi_2(x):=
\begin{cases}
\Phi_\e^1(-x), \qquad & x\in [a,h_1-x_2^\e], \\
\Phi_\e^{i+1}((-1)^{i+1}x), &  x\in [h_i+y_{i+1}^\e, h_{i+1}+y_i^\e]\quad  \text{ and } i=1,\dots,N-1, \\
(-1)^{N}, & x\in [h_N+y_{N-1}^\e,b].
\end{cases}	
\end{equation*}
The proof is thereby complete. 
\end{proof}

We, again, point out that since the integrals in \eqref{int_SW} are finite only if $\theta \in(1,2)$, 
the compactons solutions constructed in Proposition \ref{prop:comp} exist only in such a case.

 \subsection{Periodic solutions  for $\theta \geq 2$} 
As already mentioned in the previous section, when $\theta \geq 2$ the integrals in \eqref{int_SW} diverge, so that solutions to \eqref{eq:comp} cannot touch  the values $\pm 1$ and compactons solutions do not exist anymore. 
In this case, we construct a different type of stationary solutions with a transition layer structure, 
which can be seen as a restriction of {\it periodic solutions} on the whole real line. 
The study of all periodic solutions to \eqref{eq:staz} is beyond the scope of the paper and it is strictly connected to the specific form of the potential $F$ which, 
in our case, is a very generic function satisfying \eqref{ipoF1}-\eqref{ipoF2} and may give raise to infinitely many kinds of periodic solutions in the whole real line,
see two examples in Figures \ref{fig3}-\ref{fig4} below.
Indeed, assumptions \eqref{ipoF1}-\eqref{ipoF2} only assure that $F$  is a double well potential with wells of equal depth in $\pm 1$, 
and describe its behavior close to these minimal points, while give no informations of the shape of $F$ between them.  
Here, we are interested in periodic solutions oscillating between values close to $\pm1$ and assumptions \eqref{ipoF1}-\eqref{ipoF2} are enough to prove their existence.
Denote by
\begin{equation}\label{eq:Gamma}
	\mathcal{Z}:=\{u\in(-1,1)\, :\, F'(u)=0\} \qquad \mbox{ and } \qquad \Gamma:=\min_{\mathcal Z}F(u)>0.
\end{equation}  
As it was mentioned before, assumptions \eqref{ipoF1}-\eqref{ipoF2} give no information on the structure of $\mathcal{Z}$, that is the set of all the critical points of $F$ inside $(-1,1)$,
that could be a discrete set or even an interval.
Multiplying \eqref{eq:staz} by $\varphi'$, we deduce that
\begin{equation}\label{eq:periodic}
	P_\e(\varphi')-F(\varphi)= C, \qquad \qquad \mbox{ in } \R,
\end{equation}
where $C$ is an appropriate integration constant.
For instance, if $C=0$ we obtain the constant solutions $\pm1$ or the standing wave constructed in Proposition \ref{ESW}; 
in the next result, we prove that the choice $C\in\left(-\Gamma,0\right)$ gives raise to periodic (bounded) solutions.

\begin{prop}\label{periodicR}
Let $\e>0$, $P_\e$ as in \eqref{def:P}, $F$ satisfying \eqref{ipoF1}-\eqref{ipoF2} with $\theta\geq2$ and $C\in\left(-\Gamma,0\right)$, where $\Gamma$ is defined in \eqref{eq:Gamma}. 
Denote by $\bar s\in(0,1)$ the unique number such that $-F(\bar s)=C$.
Then there exists $\e_0>0$ such that, for any $\e\in(0,\e_0)$, problem \eqref{eq:periodic} admits periodic solutions $\Phi_{T_\e}$, 
oscillating between $-\bar s$ and $\bar s$, with fundamental period $2T_\e$, where
\begin{equation}\label{period}
	T_\e(\bar s):=\int_{-\bar s}^{\bar s} \frac{ds}{J_\e\left( F(s)-F(\bar s)\right)},
\end{equation}
with $J_\e$ defined in \eqref{eq:J_e}.
\end{prop}

\begin{proof}
Fix $C\in\left(-\Gamma,0\right)$ and notice that the assumption \eqref{ipoF1} together with definition \eqref{eq:Gamma} 
imply the existence of a unique $s\in(0,1)$ such that $-F(\bar s)=C$.
Since any periodic solution is unique up to translation we can assume, without loss of generality,  that $\Phi_{T_\e}(0)=-\bar s$.
By using \eqref{eq:J_e}, from equation \eqref{eq:periodic} we can infer that $\Phi_{T_\e}$ is implicitly defined as
\begin{equation*}
	\int_{-\bar s}^{\Phi_{T_\e(x)}} \frac{du}{J_\e\left( F(u)+C\right)}=x,
\end{equation*}
provided that $F(u)+C \in [0,\ell\e^{-2}])$, where $\ell$ is defined in \eqref{def:ell}. 
Hence, we need to require
\begin{equation*}
	\max_{u \in [-\bar s,\bar s]}  F(u)+C \leq \ell \e^{-2},
\end{equation*}
and, since $C<0$, such condition is satisfied again as soon as \eqref{eq:maxF} holds.
We now have to verify the convergence of the improper integral
\begin{equation}\label{eq:int-period}
	\int_{-\bar s}^{\bar s}  \frac{ds}{J_\e\left(F(s)-F(\bar s)\right)}.
\end{equation}
By using \eqref{eq:J-important} and the Taylor expansion $F(s)-F(\bar s)= F'(\bar s)(s-\bar s)+o\left(|s-\bar s|\right)$, with $F'(\bar s)<0$ because $\bar s\in(0,1)$, we deduce
\begin{equation*}
	\int_{0}^{\bar s}  \frac{ds}{J_\e\left( F(s)-F(\bar s)\right)}\sim
	\int_{0}^{\bar s}  \frac{ds}{\sqrt{F'(\bar s)(s-\bar s)}}<+\infty.
\end{equation*}
Similarly, one can prove that 
\begin{equation*}
	\int_{-\bar s}^{0}  \frac{ds}{J_\e\left( F(s)-F(\bar s)\right)}\sim
	\int_{-\bar s}^{0}  \frac{ds}{\sqrt{F'(-\bar s)(s+\bar s)}}<+\infty,
\end{equation*}
and, as a consequence, the improper integral \eqref{eq:int-period} is finite.
Therefore, we have constructed a solution $\Phi_{T_\e}:[0, T_\e(\bar s)]$, satisfying
$$\Phi_{T_\e}(0)=-\bar s, \quad \Phi_{T_\e}\left({T_\e(\bar s)}\right)=\bar s \qquad \mbox{and} \qquad  \Phi'_{T_\e}(0)=\Phi'_{T_\e}\left({T_\e(\bar s)}\right)=0.$$
Let us now define
$$\Phi_{T_\e}(x)=\Phi_{T_\e}\left(T_\e(\bar s)-x\right), \qquad \qquad x\in\left[{T_\e(\bar s)},2T_\e(\bar s)\right].$$
It is easy to check that $\Phi_{T_\e}$ solves \eqref{eq:periodic} in  $\left[{T_\e(\bar s)},2T_\e(\bar s)\right]$ and
$$\Phi_{T_\e}\left(T_\e(\bar s) \right)=-\bar s, \qquad \Phi'_{T_\e}\left(T_\e(\bar s) \right)=0.$$
We have thus extended the solution in the interval  $\left[{T_\e(\bar s)},2T_\e(\bar s)\right]$, and thus constructed a solution in $\left[0,2T_\e(\bar s)\right]$; by iterating the same argument, we can extend the solution to the whole real line by $2T_\e(\bar s) - $periodicity, and the proof is complete.  
\end{proof}

We now make use of the solutions constructed in Proposition \ref{periodicR} to construct solutions to \eqref{eq:comp} having $N$ equidistant zeroes in $[a,b]$.

\begin{prop}\label{periodic:bounded}
Let $\e>0$, $Q$ satisfying \eqref{eq:Q-ass1}-\eqref{eq:Q-ass2}, $F$ satisfying \eqref{ipoF1}-\eqref{ipoF2} with $\theta\geq2$ and let us fix $N \in \N$. 
Then, if $\e\in(0,\bar\e)$ with $\bar\e>0$ sufficiently small, there exists $\bar s$ close to (and strictly less than) $+1$, such that 
problem \eqref{eq:comp} admits a solution oscillating between $\pm\bar s$ and with $N$ equidistant zeroes inside the interval $[a,b]$, located at 
\begin{equation}\label{locationzeros}
	h_1= \frac{T_\e(\bar s)}{2}+a \qquad \mbox{and} \qquad h_i= h_{i-1}+T_\e(\bar s), \quad i=2, \dots, N,
\end{equation}
where $T_\e(\bar s)$ is defined in \eqref{period}.
\end{prop}

\begin{proof}
The solution we are looking for can be constructed by shifting and modifying  properly $\Phi_{T_\e}$, the periodic solution of Proposition \ref{periodicR}. 
Hence, in order to construct a  solution $\psi$ with exactly $N$ zeroes in $[a,b]$ such that $\psi'(a)=\psi'(b)=0$, we proceed as follows: 
first of all, we define $\psi(x)= \Phi_{T_\e}(x-a)$ (recall that $\Phi_{T_\e}(0)=-\bar s$). In such a way $\psi(a)=-\bar s$, and we also have $\psi'(a)=0$.
Moreover, the first zero $h_1$ of $\psi$ is located exactly at $T_\e/2+a$, while $h_2=T_\e + h_1$, $h_3= T_\e+h_2$ and so on, leading to \eqref{locationzeros}. 
Thus, in order to have $b$  located in  the middle point after the last zero of $\psi$ (so that, consequently, $\psi'(b)=0$), 
we have to choose $\bar s$ in such a way that $b=N T_\e(\bar s)+a$, if $\e\in(0,\bar\e)$. 
In other words, we have to prove that, if $\e\in(0,\bar\e)$, for any $N \in \N$ and $a,b \in \R$ there exists $\bar s\approx1$ such that
\begin{equation*}
 	T_\e(\bar s)= \frac{b-a}{N}.
\end{equation*}
To this purpose, by putting the expansion \eqref{eq:J-important} in the definition \eqref{period}, we infer 
$$T_\e(\bar s)=\e\sqrt{\frac{Q'(0)}{2}}\int_{-\bar s}^{\bar s}\frac{ds}{\sqrt{F(s)-F(\bar s)}}+o(\e).$$
Thus, we have to find $\bar s\approx1$ such that
\begin{equation}\label{eq:inutile}\int_{-\bar s}^{\bar s}\frac{ds}{\sqrt{F(s)-F(\bar s)}}=\frac{\sqrt2(b-a)}{\e \sqrt{Q'(0)}N}+\frac{o(\e)}{\e}.
\end{equation}
Since the integral on the left hand side is a monotone function of $\bar s\approx1$ \cite{Carr-Pego} and it satisfies 
$$\lim_{\bar s \to 1^- }\int_{-\bar s}^{\bar s}\frac{ds}{\sqrt{F(s)-F(\bar s)}}= +\infty,$$
we can state that if $\e$ is sufficiently small, then there exists a unique $\bar s\approx1$ such that $T_\e(\bar s)= (b-a)/N$, and the proof is complete.
\end{proof}

\begin{rem}
We point out that in the case $\theta\in(1,2)$ we proved existence of compactons with a generic number $N\in\mathbb{N}$ of layers 
located at arbitrary positions $h_1,h_2,\dots,h_N$ in the interval $[a,b]$. 
On the other hand, in Proposition \ref{periodic:bounded}, we proved that, for $\theta\geq 2$, there exist solutions with $N$ layers, which oscillate among the values $\pm \bar s$, 
with $\bar s\approx1$ determined by the condition $T_\e(\bar{s})= (b-a)/N$, but the layers position is given by \eqref{locationzeros}, and so it is not random.  Nevertheless, such result can be proved also when $\theta \in (1,2)$, but only if $\e$ is large enough: indeed, one has
$$\lim_{\bar s \to 1^-} \int_{-\bar s}^{\bar s} \frac{ds}{\sqrt{ F(s)-F(\bar s)}} < +\infty \qquad \mbox{if} \qquad \theta <2,$$
and one can easily see that the right hand side of \eqref{eq:inutile} diverges if $\e \to 0^+$.

Finally, notice that we proved that for $\e\in(0,\bar\e)$ there exists a unique $\bar s$ (depending on $\e$) such that $T_\e(\bar s)= (b-a)/N$ and, in particular,
$\displaystyle{\lim_{\e\to0^+}}\bar s=1,$
meaning that, the more $\e$ is small, the closer $\bar s$ is to $1$. 
\end{rem}

\subsubsection*{Particular cases of potential $F$.} 
In order to give a hint of what can happens for particular choices of the potential $F$ satisfying 
\eqref{ipoF1}-\eqref{ipoF2}, we consider two specific examples.  
In the first one, $F$ is given by \eqref{F:ex} with $\theta\geq2$ and, as a consequence, 
the admissible levels of the energy that lead to periodic (bounded) solutions are $C \in \left(-\frac{1}{2\theta},0\right)$, see Figure \ref{fig3}.

\begin{figure}[t]\label{fig3}
\begin{center}
\includegraphics[width=9cm,height=5.7cm]{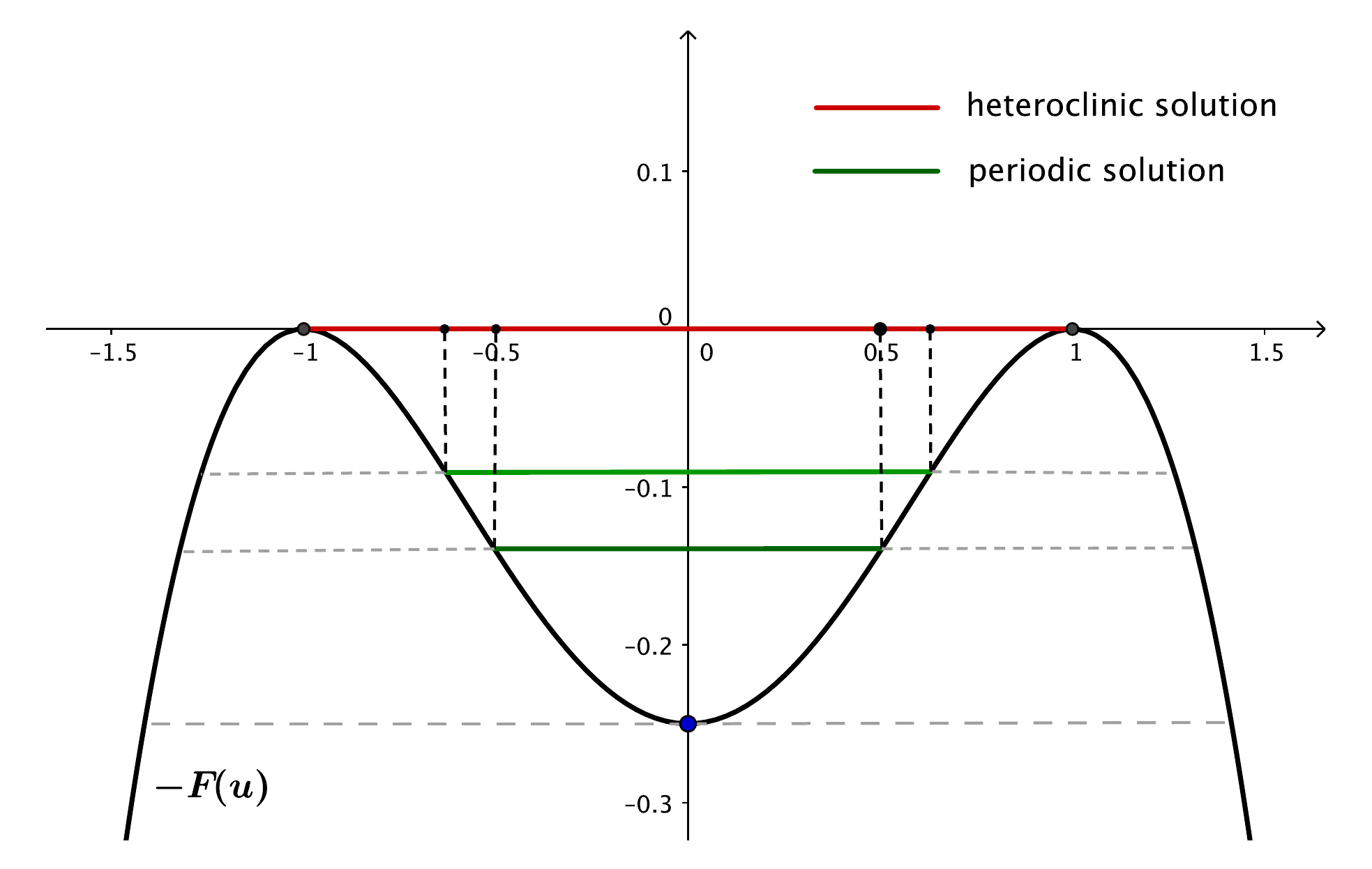}
\caption{\small{In red we depict the level $C=0$, corresponding to an heteroclinic solution; in dark green the level $C=-F(1/2)$, corresponding to the periodic solution oscillating among $\pm \frac{1}{2}$. 
Another periodic solution is depicted in green. }}
\label{fig3}
\end{center}
\end{figure}

On the other hand, let us consider a symmetric potential $F$ with a local minimum located in $u=0$ and, as a consequence, 
two local maxima located in $\pm \bar u$, for some $\bar u \in (0,1)$, see Figure \ref{fig4}, where $-F$ is depicted; in this case,
 the periodic solutions described in Proposition \ref{periodicR} appear for $C \in \left( -F(0), 0\right)$ (hence with $\Gamma=F(0)$ in \eqref{eq:Gamma}), 
 see the green line in Figure \ref{fig4}.
Moreover, if $C= -F(0)$, homoclinic solutions appear (see the blue line in Figure \ref{fig4}), 
while for $C \in \left( -F(\bar u), F(0)\right)$ one can construct new periodic solutions entirely contained either in the negative or in the positive half plane.
 Of course the case of a non-symmetric potential will be even more difficult (for instance, one will have several level of the energy corresponding to different homoclinic solutions), and this study will be the object of further investigations. 

\begin{figure}[t]\label{fig4}
\begin{center}
\includegraphics[width=11cm,height=5cm]{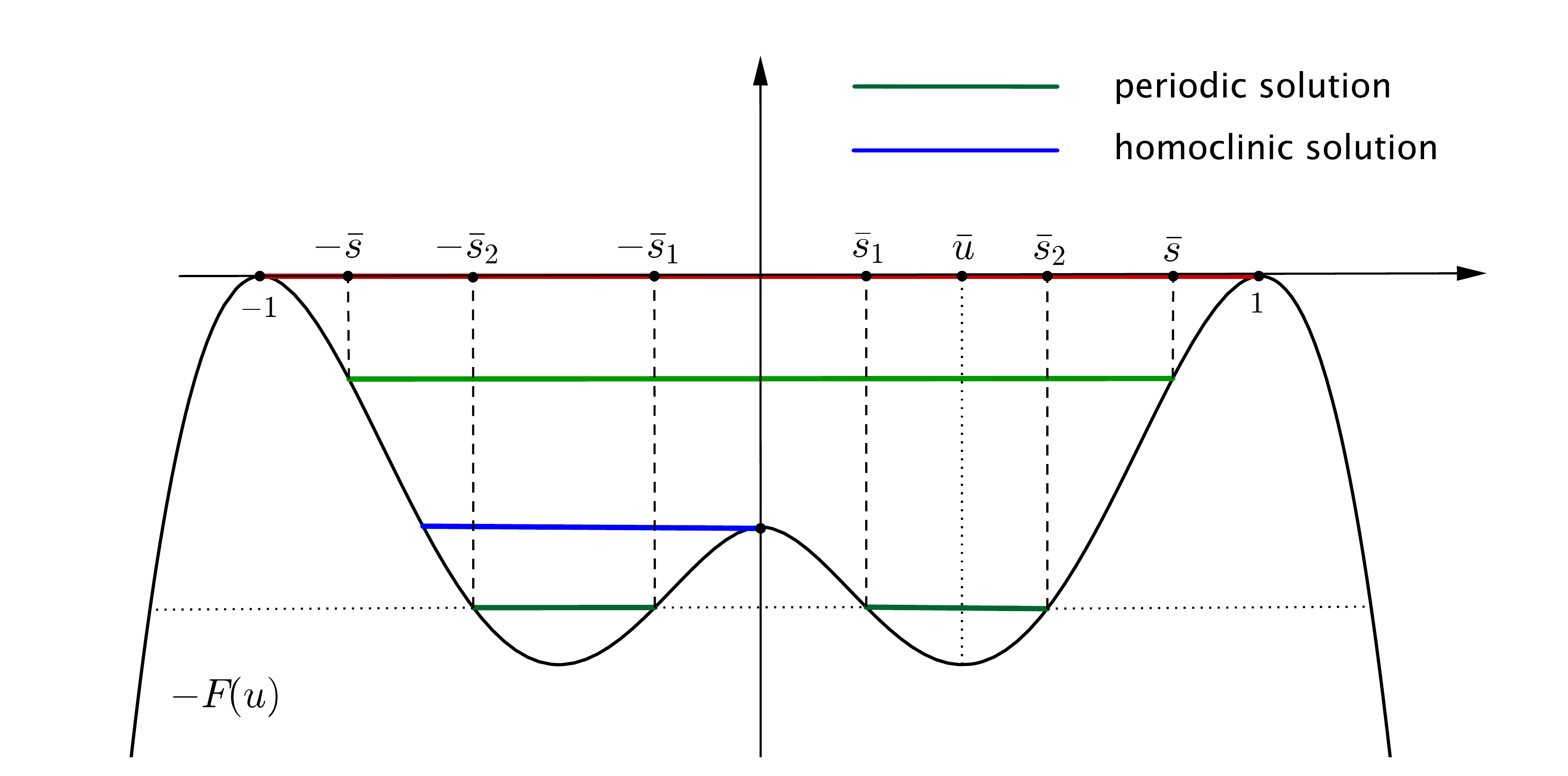}
\caption{\small{In blue we depicted the level of the energy corresponding to the homoclinic solution; in dark green the level of the energy corresponding to the periodic solutions oscillating, respectively, among $-\bar s_2, -\bar s_1 <0$ and  $\bar s_1, \bar s_2 >0$. }}
\label{fig4}
\end{center}
\end{figure}

\section{Variational results}\label{sec:energy}
In this section we collect and prove some variational results needed in order to show the slow motion phenomena of the solutions to \eqref{eq:Q-model}-\eqref{eq:Neu} in the case $\theta \geq 2$,
whose analysis will be performed in Section \ref{sec:slow}. The idea is to apply the strategy firstly developed by Bronsard and Kohn in  \cite{Bron-Kohn}, subsequently improved by Grant in \cite{Grant} and successfully used  in many other models, see for instance \cite{FPS,FPS-DCDS,FS} and references therein.

\subsection{Lyapunov functional}\label{Lyapunov}
We start by introducing the energy associated to \eqref{eq:Q-model}-\eqref{eq:Neu}
\begin{equation}\label{eq:energy}
	E_\e[u]:=\int_a^b \left[\frac{\tilde{Q}(\e^2u_x)}{\e^3}+\frac{F(u)}{\e}\right]\, dx,
\end{equation}
where $\tilde Q$ is defined in \eqref{def:ell}; we first prove that \eqref{eq:energy}  is a Lyapunov functional for the model \eqref{eq:Q-model}-\eqref{eq:Neu}, that is a functional whose time derivative is negative if  computed along the solutions to \eqref{eq:Q-model}-\eqref{eq:Neu}. 

\begin{lem}\label{lem:energy}
Let $u\in C([0,T],H^2(a,b))\cap C^1([0,T],H^1(a,b))$ be a solution of \eqref{eq:Q-model}-\eqref{eq:Neu}.
Let $E_\e$ be the functional defined in \eqref{eq:energy}. Then 
\begin{equation}\label{eq:energyestimate1}
	\frac{d}{dt}E_\e[u](t)=-\e^{-1}\int_a^b u^2_t(x,t)\, dx.
\end{equation}
and
\begin{equation}\label{eq:energyestimate}
	E_\e[u](0)-E_\e[u](T)=\e^{-1}\int_0^T\!\int_a^bu_t(x,t)^2\,dx \,dt.
\end{equation}
\end{lem}
\begin{proof}
\eqref{eq:energyestimate} directly follows from \eqref{eq:energyestimate1}: indeed once \eqref{eq:energyestimate1} is proved, an integration with respect to time over the interval $[0,T]$ yields \eqref{eq:energyestimate}. 

As for the proof of \eqref{eq:energyestimate1}, by differentiating with respect to time the energy $E_\e$, we have
\begin{equation*}
	\frac{d}{dt}E_\e[u](t)=\frac1\e\int_a^b\left[Q(\e^2 u_x)u_{xt}+F'(u)u_t\right]\,dx,
\end{equation*}
where we have used that $\tilde Q'(\e^2 u_x)=Q(\e^2 u_x)$.
Integrating by parts, exploiting the boundary conditions \eqref{eq:Neu} and that $Q(0)=0$, we deduce
\begin{equation*}
	\frac{d}{dt}E_\e[u](t)=\frac1\e\int_a^b\left[-Q(\e^2 u_x)_x+F'(u)\right]u_t\,dx.
\end{equation*}
From this, since $u$ satisfies equation \eqref{eq:Q-model}, we end up with \eqref{eq:energyestimate1}, thus completing the proof.
\end{proof}

\begin{rem}
The equality \eqref{eq:energyestimate} plays a crucial role in Section \ref{sec:slow}, where we analyze the slow motion of some solutions of \eqref{eq:Q-model}-\eqref{eq:Neu}.
Then, in our analysis we can only consider sufficiently regular solutions. 
However, it is important to notice that all the stationary solutions constructed in Section \ref{stationary} are such that the diffusion coefficient $Q'(\e^2u_x)$ is strictly positive 
(see \eqref{eq:J_e} and \eqref{eq:Fi-firstderivative}), namely the solutions are defined in the region where \eqref{eq:Q-model} is parabolic.
Similarly, all the solutions we consider in the rest of the paper satisfy the same estimates and, as a consequence, 
they have the necessary regularity to apply Lemma \ref{lem:energy}. 
\end{rem}

\subsection{Lower bounds.}
The aim of this subsection is to prove some lower bounds for the energy $E_\e$, defined in \eqref{eq:energy}, associated to a function which is sufficiently close in $L^1$-sense to a jump function with constant values $-1$ and $+1$ (we refer the reader to Definition \ref{vstruct}). 
Such variational results present a different nature  depending on either $\theta=2$ or $\theta>2$; moreover, we underline that in their proof the equation \eqref{eq:Q-model} does not come into play, unlike the result contained in Subsection \ref{Lyapunov}. 

\subsubsection{A crucial inequality} The first tool we need to prove the aforementioned lower bounds is an inequality involving the functions $Q$ and $J_\e$. 
To better understand the motivation behind such a tool,
we recall the equation which identifies the standing waves solutions, that is 
\begin{equation*}
	P_\e(\Phi'_\e)= F(\Phi_\e),
\end{equation*}
which in turn, using \eqref{defalternativa}, can be rewritten as follows 
\begin{equation}\label{uguag}
	\Phi'_\e Q(\e^2 \Phi'_\e) - \frac{1}{\e^2}\tilde Q (\e^2 \Phi'_\e)= F(\Phi_\e).
\end{equation}
We now observe that \eqref{eq:J_e} implies $\Phi'_\e=J_\e (F(\Phi_\e))$, so that substituting into \eqref{uguag}, we arrive at
\begin{equation}\label{eq:impforenergy}
	\frac{\tilde 	Q (\e^2 \Phi'_\e)}{\e^3} + \frac{F(\Phi_\e)}{\e} = \frac{\Phi'}{\e} Q(\e^2J_\e (F(\Phi_\e))).
\end{equation}
Therefore, we seek a suitable inequality such that in some sense the equality holds along the standing wave solutions. 
Inspired by the previous considerations, we state and prove the following lemma.
\begin{lem}
Let $\e,L>0$. 
If $Q\in C^1(\R)$ satisfies \eqref{eq:Q-ass1}-\eqref{eq:Q-ass2} and $J_\e$ satisfies \eqref{eq:J_e}, then
\begin{equation}\label{eq:strangeineq+}
	\frac{\tilde Q (\e^2 x)}{\e^3}+\frac{y}{\e}\geq \frac{|x|}{\e} \, Q(\e^2J_\e (y)),
\end{equation}
for any $(x,y)\in[-\kappa \e^{-2}, \kappa \e^{-2}]\times[0,\ell \e^{-2}]$, where $\kappa$ and $\ell$ are defined in \eqref{eq:Q-ass2} and \eqref{def:ell}, respectively.
\end{lem}
\begin{proof}
Since, by assumption, $\tilde Q$ is even, in order to prove \eqref{eq:strangeineq+} it is sufficient to study the sign of the function
\begin{equation*}
	g(x,y):=\tilde Q (\e^2 x)+\e^2 y -\e^2 x \, Q(\e^2 J_\e (y)),
\end{equation*}
for all $ x \in [0,\kappa \e^{-2}]$ and for all $ y \in [0,\ell\e^{-2}]$.
For any $(x,y)$ belonging to the inside of such a rectangle we have
\begin{align*}
	& g_x(x,y)=\e^2  \,  \tilde Q'(\e^2 x)-\e^2 \, Q(\e^2 _\e J(y)) = 0 \qquad \mbox{ if and only if } \qquad Q(\e^2 x)= Q(\e^2 J_\e (y)).
\end{align*}
Since $Q$ is strictly increasing in the interval $[0,\kappa \e^{-2}]$, we thus have $g_x(x,y)=0$ if and only if $x=J_\e (y)$. 
Let us now evaluate $g_y$ in such points:
\begin{align*}
	g_y(x,y) \Big|_{x= J_\e (y)}&=\e^2 -\e^2 x \, Q'(\e^2  J_\e (y)) \,  \e^2 J'_\e (y)\Big|_{x=J_\e (y)} \\
	&=\e^2 -\e^4 x \, \frac{Q'(\e^2 J_\e (y))}{P'_\e\left( J_\e (y)\right)}\Big|_{x=J_\e (y)} \\
	&= \e^2 -\frac{\e^2 x}{P_\e'(x)} \, Q'(\e^2 x) = \e^2 -\frac{\e^4 x}{\e^2 x \,  Q'(\e^2 x)} \, Q'(\e^2 x)=0,
\end{align*}
where we used \eqref{def:P} and \eqref{eq:J_e}. 
It follows that the only {\it internal} critical points of the function $g$ are given by $x=J_\e (y)$, and we have
\begin{align*}
	g(x,y) \Big|_{x= J_\e (y)} &= \tilde Q(\e^2 x) + \e^2 P_\e(x) -\e^2 x \, Q(\e^2 x) \\
	&=\tilde Q(\e^2 x)  + \e^2 \left[ x \, Q(\e^2 x) -\frac{1}{\e^2} \tilde Q(\e^2 x)\right] -\e^2 x \, Q(\e^2 x) =0.
\end{align*}
Let us now study the function $g$ on the boundary of the rectangle $[0, \kappa \e^{-2}]\times[0,\ell \e^{-2}]$, which is formed by four segments; we start with
\begin{equation*}
	g(0,y)=\tilde Q(0)+\e^2 y \geq 0 \qquad \mbox{for all} \quad  y\in[0,\ell \e^{-2}],
\end{equation*}
and
\begin{equation*}
	f(x,0)=\tilde Q(\e^2 x)-\e^2 x \, Q(\e^2J_\e (0)) =\tilde Q(\e^2 x) \geq 0  \qquad \mbox{for all} \quad  x \in[0,\kappa\e^{-2}],
\end{equation*}
where we used the fact that $J_\e(0)=0$ and $Q(0)=0$.
Next, for $y \in [0, \ell \e^{-2}]$ we consider the function
\begin{equation*}
	g_1(y):=g(\kappa \e^{-2},y)=\tilde Q (\kappa)+\e^2 y-\kappa Q(\e^2 J_\e (y)),
\end{equation*}
and we have
\begin{align*}
	g'_1(y)&=\e^2 -\kappa\e^2 Q'(\e^2 J_\e (y)) \frac{1}{P'_\e(J_\e (y))}\\
	 &=\e^2 -\frac{\kappa \e^2  Q'(\e^2 J_\e (y))}{\e^2 J_\e (y) Q'(\e^2 J_\e (y))} = \e^2 -\frac{\kappa}{J_\e (y)}.
\end{align*}
Recalling \eqref{def:ell}, we observe that the function $g_1$ is such that
\begin{equation*}
	g_1'(y) \leq 0 \quad \forall \ y \in  [0, \ell \e^{-2}] \qquad \mbox{and} \qquad \lim_{y \to 0} g_1'(y) = -\infty.
\end{equation*}
Moreover $g_1(0)= \tilde Q(\kappa) >0$ while
\begin{align*}
	g_1(\ell \e^{-2}) = g_1 \left( P_\e(\kappa \e^{-2})\right) &= \tilde Q(\kappa) + \e^2 P_\e(\kappa \e^{-2}) - \kappa Q(\kappa)  \\
	&=\tilde Q(\kappa) + \ell - \kappa Q(\kappa) \\
	&= \tilde Q(\kappa)+ \kappa Q(k) - \tilde Q(\kappa)- \kappa Q(\kappa)=0,
\end{align*}
implying that $g(\kappa \e^{-2},y) \geq 0$ for all $y \in [0, \ell \e^{-2}]$. Finally,  for $x \in [0,\kappa \e^{-2}]$ we consider
\begin{align*}
	g_2(x):=g(x,\ell \e^{-2})&=\tilde Q ( \e^2 x)+\ell- \e^2 x Q(\e^2 J_\e ( \ell \e^{-2})) \\
	&=\tilde Q ( \e^2 x) + \kappa Q(\kappa)-\tilde Q(\kappa)-\e^2 x Q(\kappa),
\end{align*}
where in the last equality we used \eqref{def:ell}-\eqref{eq:J_e}, which imply $J_\e(\ell \e^{-2}) = \kappa \e^{-2}$. 
Going further, we have
\begin{equation*}
	g_2(0)= \ell >0, \qquad \qquad 
	g_2(\kappa \e^{-2}) = \tilde Q ( \kappa)+ \kappa Q(\kappa)-\tilde Q(\kappa)-\kappa Q(\kappa)=0,
\end{equation*}
and 
\begin{align*}
	g_2'(x)= \e^2 \tilde Q' (\e^2x )-\e^2 Q(k) = \e^2 \left( Q(\e^2 x)-Q(k)\right) \leq 0,
\end{align*}
since $Q$ is increasing and $x \leq k\e^{-2}$. Hence, $g(x, \ell \e^{-2}) \geq 0$ for all $x \in [0, \kappa \e^{-2}]$.

We thus proved that $g$ is non negative on the boundary; since $g=0$ at the only internal critical points,
we have that $g$ is non negative for all $(x,y)\in [0,\kappa \e^{-2}]\times[0,\ell\e^{-2}]$, and the proof is complete.
\end{proof}

The inequality \eqref{eq:strangeineq+} is crucial because it allows us to state that if $\bar{u}$ is a monotone function
connecting the two stable points $+1$ and $-1$ and \eqref{eq:maxF} holds true, then the energy \eqref{eq:energy} satisfies
\begin{equation}\label{eq:c_eps}
	E_\e[\bar{u}]\geq\int_a^b\frac{|\bar{u}'|}\e Q\left(\e^2 J_\e(F(\bar{u}))\right)\,dx=\e^{-1}\int_{-1}^{+1}Q\left(\e^2 J_\e(F(s))\right)\,ds=:c_\e.
\end{equation}
Our next goal is to show that the positive constant $c_\e$ defined in \eqref{eq:c_eps} represents the minimum energy to have a {\it single} transition between $-1$ and $+1$; 
having this in mind, we fix once for all $N\in\mathbb{N}$ and a {\it piecewise constant function} $v$ with $N$ transitions as follows:
\begin{equation}\label{vstruct}
	\begin{aligned}
	v:[a,b]\rightarrow\{-1,+1\}\  \hbox{with $N$ jumps located at $a<h_1<h_2<\cdots<h_N<b$ and}\ r>0\\
	\hbox{such that}\ (h_i-r,h_i+r)\cap(h_j-r,h_j+r)=\emptyset \  \hbox{for}\ i\neq j\ \hbox{and}\   a\leq h_1-r,\ h_N+r\leq b.
	\end{aligned}
\end{equation}
The aforementioned lower bounds will allow us to state that if $\{ u^\e\}_{\e >0}$ is a family of functions sufficiently close to $v$ in $L^1$, then 
$$E_\e[{u^\e}] \geq N c_\e-R_{\theta,\e},$$ 
where the reminder term $R_{\theta,\e}$ goes to zero as $\e \to 0^+$ with a speed rate depending on $\theta$.

\subsubsection{Lower bound in the critical case $\theta=2$.}
Let us start by proving the lower bound in the case $\theta=2$, where the reminder term $R_{\theta,\e}$ is exponentially small as $\e\to0^+$.  
\begin{prop}\label{prop:lower}
Assume that $Q\in C^1(\R)$ satisfies \eqref{eq:Q-ass1}-\eqref{eq:Q-ass2} and that $F\in C^1(\R)$ satisfies \eqref{ipoF1}-\eqref{ipoF2} with $\theta=2$. 
Let us set 
\begin{equation}\label{eq:mathc-Q}
	\mathcal{Q}:=\max_{s\in[-\kappa,\kappa]}Q'(s),
\end{equation}
where $\kappa$ is given in \eqref{eq:Q-ass2}.
Moreover, let $v$ be as in \eqref{vstruct} and $A\in(0,r\sqrt{2\lambda_1 \mathcal{Q}^{-1}})$ with
$\lambda_1>0$ (independent on $\e$) as in \eqref{ipoF2}.
Then, there exist $\e_0,C,\delta>0$ (depending only on $Q,F,v$ and $A$) such that if $u\in H^1(a,b)$ satisfies 
\begin{equation}\label{eq:u-v}
	\|u-v\|_{{}_{L^1}}\leq\delta,
\end{equation}
then for any $\e\in(0,\e_0)$,
\begin{equation}\label{eq:lower}
	E_\varepsilon[u]\geq Nc_\e-C\exp(-A/\varepsilon),
\end{equation}
where $E_\e$ and $c_\e$ are defined in \eqref{eq:energy} and \eqref{eq:c_eps}, respectively.
\end{prop}
\begin{proof}
Fix $u\in H^1(a,b)$ satisfying \eqref{eq:u-v} and $\e$ such that \eqref{eq:maxF} holds true.
Take $\hat r\in(0,r)$ so small that 
\begin{equation}\label{eq:nu}
	A\leq(r-\hat r)\sqrt{2\mathcal Q^{-1}\lambda_1}.
\end{equation}
Then, choose $0<\rho <\eta$ (with $\eta$ given by \eqref{ipoF2}) sufficiently small that
\begin{equation}\label{eq:forrho2}
\begin{aligned}
	\int_{1-\eta}^{1-\rho}Q(\e^2J_\e(F(s)))\,ds&>\int_{1-\rho}^{1}Q(\e^2J_\e(F(s)))\,ds,  \\
	\int_{-1+\rho}^{-1+\eta}Q(\e^2J_\e(F(s)))\,ds&> \int_{-1}^{-1+\rho}Q(\e^2J_\e(F(s)))\,ds.
	\end{aligned}
\end{equation}
Let us focus our attention on $h_i$, one of the discontinuous points of $v$ and, to fix ideas, 
let $v(h_i\pm r)=\pm1$, the other case being analogous.
We claim that assumption \eqref{eq:u-v} implies the existence of $r_+$ and $r_-$ in $(0,\hat r)$ such that 
\begin{equation}\label{2points}
	|u(h_i+r_+)-1|<\rho, \qquad \quad \mbox{ and } \qquad \quad |u(h_i-r_-)+1|<\rho.
\end{equation}
Indeed, assume by contradiction that $|u-1|\geq\rho$ throughout $(h_i,h_i+\hat r)$; then
\begin{equation*}
	\delta\geq\|u-v\|_{{}_{L^1}}\geq\int_{h_i}^{h_i+\hat r}|u-v|\,dx\geq\hat r\rho,
\end{equation*}
and this leads to a contradiction if we choose $\delta\in(0,\hat r\rho)$.
Similarly, one can prove the existence of $r_-\in(0,\hat r)$ such that $|u(h_i-r_-)+1|<\rho$.

Now, we consider the interval $(h_i-r,h_i+r)$ and claim that
\begin{equation}\label{eq:claim}
	\int_{h_i-r}^{h_i+r}\left[\frac{\tilde{Q}(\e^2u_x)}{\e^3}+\frac{F(u)}{\e}\right]\,dx\geq c_\e-\tfrac{C}N\exp(-A/\varepsilon),
\end{equation}
for some $C>0$ independent on $\e$.
Observe that from \eqref{eq:strangeineq+}, it follows that for any $a\leq c<d\leq b$,
\begin{equation}\label{eq:ineq}
	\int_c^d\left[\frac{\tilde{Q}(\e^2u_x)}{\e^3}+\frac{F(u)}{\e}\right]\,dx \geq\e^{-1}\left|\int_{u(c)}^{u(d)}Q(\e^2J_\e(F(s)))\,ds\right|.
\end{equation}
Hence, if $u(h_i+r_+)\geq1$ and $u(h_i-r_-)\leq-1$, then from \eqref{eq:ineq} we can conclude that
\begin{equation*}
	\int_{h_i-r_-}^{h_i+r_+}\left[\frac{\tilde{Q}(\e^2u_x)}{\e^3}+\frac{F(u)}{\e}\right]\,dx\geq c_\e,
\end{equation*}
which implies \eqref{eq:claim}.
On the other hand, notice that in general we have
\begin{align}
	\int_{h_i-r}^{h_i+r}\left[\frac{\tilde{Q}(\e^2u_x)}{\e^3}+\frac{F(u)}{\e}\right]\,dx & 
	\geq \int_{h_i+r_+}^{h_i+r}\left[\frac{\tilde{Q}(\e^2u_x)}{\e^3}+\frac{F(u)}{\e}\right]\,dx\notag \\ 
	& \quad + \int_{h_i-r}^{h_i-r_-}\left[\frac{\tilde{Q}(\e^2u_x)}{\e^3}+\frac{F(u)}{\e}\right]\,dx \notag \\
	& \quad +\e^{-1}\int_{-1}^{1}Q(\e^2J_\e(F(s)))\,ds\notag\\
	&\quad-\e^{-1}\int_{-1}^{u(h_i-r_-)}Q(\e^2J_\e(F(s)))\,ds \notag \\
	& \quad-\e^{-1}\int_{u(h_i+r_+)}^{1}Q(\e^2J_\e(F(s)))\,ds\notag \\
	&=:I_1+I_2+c_\e-\alpha_\e-\beta_\e, \label{eq:Pe}
\end{align}
where we again used \eqref{eq:ineq}.
Let us estimate the first two terms of \eqref{eq:Pe}. 
Regarding $I_1$, assume that $1-\rho<u(h_i+r_+)<1$ and consider the unique minimizer $z:[h_i+r_+,h_i+r]\rightarrow\R$ 
of $I_1$ subject to the boundary condition $z(h_i+r_+)=u(h_i+r_+)$.
If the range of $z$ is not contained in the interval $(1-\eta,1+\eta)$, then from \eqref{eq:ineq}, it follows that
\begin{equation}\label{E>fi}
	\int_{h_i+r_+}^{h_i+r}\left[\frac{\tilde{Q}(\e^2z')}{\e^3}+\frac{F(z)}{\e}\right]\,dx>\e^{-1}\int_{u(h_i+r_+)}^{1}Q(\e^2J_\e(F(s)))\,ds=\beta_\e,
\end{equation}
by the choice of $r_+$ and $\rho$, see \eqref{eq:forrho2}. 
Suppose, on the other hand, that the range of $z$ is contained in the interval $(1-\eta,1+\eta)$. 
Then, the Euler-Lagrange equation for $z$ is
\begin{align*}
	&\e Q'(\e^2z'(x))z''(x)=\e^{-1}F'(z(x)), \quad \qquad x\in(h_i+r_+,h_i+r),\\
	&z(h_i+r_+)=u(h_i+r_+), \quad \qquad z'(h_i+r)=0.
\end{align*}
Denoting by $\psi(x):=(z(x)-1)^2$, we have $\psi'=2(z-1)z'$ and 
\begin{equation*}
	\psi''(x)=2(z(x)-1)z''(x)+2z'(x)^2\geq\frac{2}{\mathcal{Q}\varepsilon^2}(z(x)-1)F'(z(x)),
\end{equation*}
where we used the fact that $\e^2|z'|\leq\kappa$ (see \eqref{eq:Fi-firstderivative}) and \eqref{eq:mathc-Q}.
Since $|z(x)-1|\leq\eta$ for any $x\in[h_i+r_+,h_i+r]$, using \eqref{ipoF2} with $\theta=2$, we obtain
\begin{equation*}
	\psi''(x)\geq\frac{2\lambda_1}{\mathcal{Q}\varepsilon^2}(z(x)-1)^2\geq \frac{\mu^2}{\varepsilon^2}\psi(x),
\end{equation*}
where $\mu=A/(r-\hat r)$ and we used \eqref{eq:nu}. 
Thus, $\psi$ satisfies
\begin{align*}
	\psi''(x)-\frac{\mu^2}{\varepsilon^2}\psi(x)\geq0, \quad \qquad x\in(h_i+r_+,h_i+r),\\
	\psi(h_i+r_+)=(u(h_i+r_+)-1)^2, \quad \qquad \psi'(h_i+r)=0.
\end{align*}
We compare $\psi$ with the solution $\hat \psi$ of
\begin{align*}
	\hat\psi''(x)-\frac{\mu^2}{\varepsilon^2}\hat\psi(x)=0, \quad \qquad x\in(h_i+r_+,h_i+r),\\
	\hat\psi(h_i+r_+)=(u(h_i+r_+)-1)^2, \quad \qquad \hat\psi'(h_i+r)=0,
\end{align*}
which can be explicitly calculated to be
\begin{equation*}
	\hat\psi(x)=\frac{(u(h_i+r_+)-1)^2}{\cosh\left[\frac\mu\varepsilon(r-r_+)\right]}\cosh\left[\frac\mu\varepsilon(x-(h_i+r))\right].
\end{equation*}
By the maximum principle, $\psi(x)\leq\hat\psi(x)$ so, in particular,
\begin{equation*}
	\psi(h_i+r)\leq\frac{(u(h_i+r_+)-1)^2}{\cosh\left[\frac\mu\varepsilon(r-r_+)\right]}\leq2\exp(-A/\varepsilon)(u(h_i+r_+)-1)^2.
\end{equation*}
Then, we have 
\begin{equation}\label{|z-v+|<exp}
	|z(h_i+r)-1|\leq\sqrt2\exp(-A/2\varepsilon)\rho.
\end{equation}
Thanks to the expansion
\begin{equation}\label{eq:expQ}
	Q(s)=Q'(0)s+o(s^2),
\end{equation}
and \eqref{eq:ass-F3}-\eqref{eq:J-important}, we can choose $\e>0$ small enough that 
\begin{equation}\label{eq:crucial}
	\left|Q(\e^2J_\e(F(s)))\right|\leq C\e^2 J_\e(F(s))\leq C\e\sqrt{F(s)}\leq C\e|1-s|,
\end{equation}
for any $s\in[z(h_i+r),1]$; as a consequence, \eqref{|z-v+|<exp} yields
\begin{equation}\label{fi<exp}
	\e^{-1}\left|\int_{z(h_i+r)}^{1}Q(\e^2J_\e(F(s)))\,ds\right|\leq C\exp(-A/\varepsilon). 
\end{equation}
From \eqref{eq:ineq}-\eqref{fi<exp} it follows that, for some constant $C>0$, 
\begin{align}
	\int_{h_i+r_+}^{h_i+r}\left[\frac{\tilde{Q}(\e^2z')}{\e^3}+\frac{F(z)}{\e}\right]\,dx &\geq\e^{-1}\left|\int_{z(h_i+r_+)}^{1}Q(\e^2J_\e(F(s)))\,ds\,-\right.\nonumber \\
	&\qquad \qquad\left.\int_{z(h_i+r)}^{1}Q(\e^2J_\e(F(s)))\,ds\right| \nonumber\\
	& \geq\beta_\e-\tfrac{C}{2N}\exp(-A/\varepsilon). \label{E>fi-exp}
\end{align}
Combining \eqref{E>fi} and \eqref{E>fi-exp}, we get that the constrained minimizer $z$ of the proposed variational problem satisfies
\begin{equation*}	
	\int_{h_i+r_+}^{h_i+r}\left[\frac{\tilde{Q}(\e^2z')}{\e^3}+\frac{F(z)}{\e}\right]\,dx \geq\beta_\e-\tfrac{C}{2N}\exp(-A/\varepsilon).
\end{equation*}
The restriction of $u$ to $[h_i+r_+,h_i+r]$ is an admissible function, so it must satisfy the same estimate and we have
\begin{equation}\label{eq:I1}
	I_1\geq\beta_\e-\tfrac{C}{2N}\exp(-A/\varepsilon).
\end{equation}
The term $I_2$ on the right hand side of \eqref{eq:Pe} is estimated similarly by analyzing  the interval $[h_i-r,h_i-r_-]$ 
and using the second condition of \eqref{eq:forrho2} to obtain the corresponding inequality \eqref{E>fi}.
The obtained lower bound reads:
\begin{equation}\label{eq:I2}	
	I_2\geq\alpha_\e-\tfrac{C}{2N}\exp(-A/\varepsilon).
\end{equation}
Finally, by substituting \eqref{eq:I1} and \eqref{eq:I2} in \eqref{eq:Pe}, we deduce \eqref{eq:claim}.
Summing up all of these estimates for $i=1, \dots, N$, namely for all transition points, we end up with
\begin{equation*}
	E_\varepsilon[u]\geq\sum_{i=1}^N\int_{h_i-r}^{h_i+r}\left[\frac{\tilde{Q}(\e^2u_x)}{\e^3}+\frac{F(u)}{\e}\right]\,dx\geq Nc_\e-C\exp(-A/\varepsilon),
\end{equation*}
and the proof is complete.
\end{proof}

\subsubsection{Lower bound in the supercritical case $\theta>2$.}
We now deal with the case $\theta>2$, where we have a {\it weaker} lower bound for the energy, which is stated and proved in the following proposition.
\begin{prop}\label{prop:lower_deg}
Assume that $Q\in C^1(\R)$ satisfies \eqref{eq:Q-ass1}-\eqref{eq:Q-ass2} and that $F\in C^1(\R)$ satisfies \eqref{ipoF1}-\eqref{ipoF2} with $\theta>2$. 
Let $v$ as in \eqref{vstruct} and define the sequence
\begin{equation}\label{eq:exp_alg}
	k_j:=\sum_{m=1}^j\alpha^m,
	\qquad\qquad  \mbox{ where }  \quad \alpha:=\displaystyle\frac{1}{2}+\frac{1}{\theta}.
\end{equation}
Then, for any $j\in \mathbb N$ there exist constants $\delta_j>0$ and $C>0$ such that if $u\in H^1(a,b)$ satisfies
\begin{equation}\label{|w-v|_L^1<delta_l}
	\|u-v\|_{L^1}\leq\delta_j,
\end{equation}
and 
\begin{equation}\label{E_p(w)<Nc0+eps^l}
	E_\e[u]\leq Nc_\e+C\varepsilon^{k_{j}},
\end{equation}
with $\varepsilon$ sufficiently small, then
\begin{equation}\label{E_p(w)>Nc0-eps^l}
	E_\e[u]\geq Nc_\e-C_j\varepsilon^{k_{j+1}}.
\end{equation}
\end{prop}

\begin{proof}
We prove our statement by induction on $j\geq1$.
Let us begin by considering the case of only one transition $N=1$ and
let $h_1$ be the only point of discontinuity of $v$ and assume, without loss of generality, that $v=-1$ on $(a,h_1)$. 
Also, we choose $\delta_j$ small enough such that
\begin{equation*}
	(h_1-2j\delta_j,h_1+2j\delta_j)\subset(a,b).
\end{equation*}
Our goal is to show that for any $j\in\mathbb N$ there exist $x_j\in(h_1- 2j\delta_j,h_1)$ and $y_j\in(h_1, h_{1}+2j\delta_j)$ such that
\begin{equation}\label{x_k,y_k}
	u(x_j)\leq-1+C\varepsilon^\frac{k_j+1}{\theta}, \qquad u(y_j)\geq 1-C\varepsilon^\frac{k_j+1}{\theta},
\end{equation}
and
\begin{equation}\label{E_p(w)-xk,yk}
	\int_{x_j}^{y_j}\left[\frac{\tilde{Q}(\e^2u_x)}{\e^3}+\frac{F(u)}{\e}\right]\geq c_\e-C\varepsilon^{k_{j+1}},
\end{equation}
where $\{k_j\}_{{}_{j\geq1}}$ is defined in \eqref{eq:exp_alg}.
We start with the base case $j=1$, and we show that hypotheses \eqref{|w-v|_L^1<delta_l} and \eqref{E_p(w)<Nc0+eps^l} imply the existence of two points 
$x_1{\in(h_1-2\delta_j,h_1)}$ and $y_1\in(h_1,h_1+2\delta_j)$ such that
\begin{equation}\label{x_1,y_1}
	u(x_1)\leq-1+C\varepsilon^\frac1{\theta}, \qquad u(y_1)\geq1-C\varepsilon^\frac1{\theta}.
\end{equation}
Here and throughout, $C$ represents a positive constant that is independent of $\varepsilon$,
whose value may change from line to line. 
From hypothesis \eqref{|w-v|_L^1<delta_l}, we have
\begin{equation}\label{int su (gamma,b)}
	\int_{h_1}^b|u-1|\leq\delta_j,
\end{equation}
so that, denoting by $S^-:=\{y:u(y)\leq0\}$ and by $S^+:=\{y:u(y)>0\}$, 
 \eqref{int su (gamma,b)} yields
\begin{equation*}
\begin{aligned}
\textrm{meas}(S^-\cap(h_1,b))\leq\delta_j \qquad \mbox{and} \qquad
	\textrm{meas}(S^+\cap(h_1,h_1+2\delta_j))\geq\delta_j.
	\end{aligned}
\end{equation*}
Furthermore, from \eqref{E_p(w)<Nc0+eps^l}  with $j=1$, we obtain
\begin{equation*}
	\int_{S^+\cap(h_1,h_1+2\delta_j)}\frac{F(u)}\varepsilon\, dx\leq c_\e+C\e^\alpha,
\end{equation*}
and therefore there exists $y_1\in S^+\cap(h_1,h_1+2\delta_j)$ such that
\begin{equation*}
	F(u(y_1))\leq \frac{c_\e+C\e^\alpha}{\delta_j}\varepsilon.
\end{equation*}
Since $F$ vanishes only at $\pm1$ and $u(y_1)>0$ we can choose $\e$ so small that the latter condition implies $|u(y_1)-1|<\eta$;
hence, from \eqref{eq:ass-F3}, it follows that $u(y_1)\geq 1-C\e^{\frac1{\theta}}$.
The existence of $x_1{\in S^-\cap(h_1-2\delta_j,h_1)}$ such that $u(x_1)\leq-1+C\varepsilon^\frac1{\theta}$ can be proved similarly.

Now, let us prove that \eqref{x_1,y_1} implies \eqref{E_p(w)-xk,yk} in the case $j=1$,
and as a trivial consequence we obtain the statement \eqref{E_p(w)>Nc0-eps^l} with $j=1$ and $N=1$.
Indeed, by using \eqref{eq:strangeineq+} and \eqref{x_1,y_1} one deduces 
\begin{equation*}
	\begin{aligned}
		E_\e[u]&\geq\int_{x_1}^{y_1}\left[\frac{\tilde{Q}(\e^2u_x)}{\e^3}+\frac{F(u)}{\e}\right]
		\geq\e^{-1}\int_{u(x_1)}^{u(y_1)} Q(\e^2J_\e(F(s)))\,ds\\
		&\geq c_\e-\e^{-1}\int_{1-C\varepsilon^\frac1{\theta}}^1 Q(\e^2J_\e(F(s)))\,ds-\e^{-1}\int_{-1}^{-1+C\varepsilon^\frac1{\theta}}Q(\e^2J_\e(F(s)))\,ds.
	\end{aligned}
\end{equation*}
Reasoning as in \eqref{eq:crucial} and using \eqref{eq:ass-F3}, we infer
\begin{align*}
	\left|\int_{1-C\varepsilon^\frac1{\theta}}^1 Q(\e^2J_\e(F(s)))\,ds\right|&\leq 
	C\e\int_{1-C\varepsilon^\frac1{\theta}}^1(1-s)^{\frac\theta2}\leq C\varepsilon^{\frac32+\frac1\theta},\\
	\left|\int_{-1}^{-1+C\varepsilon^\frac1{\theta}} Q(\e^2J_\e(F(s)))\,ds\right|&\leq 
	C\e\int_{-1}^{-1+C\varepsilon^\frac1{\theta}}(s+1)^{\frac\theta2}\leq C\varepsilon^{\frac32+\frac1\theta},
\end{align*}
and, as a trivial consequence, 
\begin{equation}\label{stima-l=1}
	E_\e[u]\geq c_\e-C\e^\alpha,
\end{equation}
where $\alpha$ is defined in \eqref{eq:exp_alg}.
This concludes the proof in the case $j=1$ with one transition $N=1$.

We now enter the core of the induction argument, proving that if \eqref{E_p(w)-xk,yk} holds true for for any $i\in\{1,\dots,j-1\}$, $j\geq2$, then \eqref{x_k,y_k} holds true.
By using \eqref{|w-v|_L^1<delta_l} we have
\begin{equation}\label{meas>delta_l-k}
	\textrm{meas}(S^+\cap(y_{j-1},y_{j-1}+2\delta_j))\geq\delta_j.
\end{equation}
Furthermore, by using  \eqref{E_p(w)<Nc0+eps^l} and \eqref{E_p(w)-xk,yk} in the case $j-1$, we deduce
\begin{equation*}
	\int_{y_{j-1}}^b\frac{F(u)}\varepsilon dx\leq C\varepsilon^{k_j},
\end{equation*}
implying
\begin{equation}\label{int F(w)<Ceps^k+1}
	\int_{S^+\cap(y_{j-1},y_{j-1}+2\delta_j)} F(u)\,dx\leq C\varepsilon^{k_j+1}.
\end{equation}
Finally, from  \eqref{meas>delta_l-k} and \eqref{int F(w)<Ceps^k+1} there exists $y_{j}\in S^+\cap(y_{j-1},y_{j-1}+2\delta_j)$ such that
\begin{equation*}
	F(u(y_{j}))\leq \frac{C}{\delta_j}\varepsilon^{k_j+1},
\end{equation*}
and, as a consequence, we have the existence of $y_{j}\in(y_{j-1},y_{j-1}+2\delta_j)$ as in \eqref{x_k,y_k}. 
The existence of $x_{j}\in(x_{j-1}-2\delta_j,x_{j-1})$ can be proved similarly. 

Proceeding as done to obtain \eqref{stima-l=1}, one can easily check that \eqref{x_k,y_k} implies
\begin{equation*}
	\int_{x_{j}}^{y_{j}}\left[\frac{\tilde{Q}(\e^2u_x)}{\e^3}+\frac{F(u)}{\e}\right]\,dx\geq c_\e-C\varepsilon^{\alpha(k_j+1)}.
\end{equation*}
Since the definition \eqref{eq:exp_alg} implies $\alpha(k_j+1)=k_{j+1}$, the induction argument is completed, as well as the proof in case $N=1$.

The previous argument can be easily adapted to the case $N>1$. 
Let $v$ be as in \eqref{vstruct}, and set $a=h_0, h_{N+1}=b$. 
We argue as in the case $N=1$ in each point of discontinuity $h_i$, by choosing
the constant $\delta_j$ so that 
$$
	{h_i}+2j\delta_j<h_{i+1}-2j\delta_j,\qquad \quad 0\leq i\leq N,
$$
and by assuming, without loss of generality, that $v=-1$ on $(a,h_1)$. 
Proceeding as in \eqref{x_1,y_1}, one can obtain the existence of $x^i_1\in(h_i-2\delta_j,h_i)$ and $y^i_1\in(h_i,h_i+2\delta_j)$ such that
\begin{align*}
	u(x^i_1)&\approx (-1)^i, &u(y^i_1)&\approx(-1)^{i+1},\\
	F(u(x^i_1))&\leq C\varepsilon, &F(u(y^i_1))&\leq C\varepsilon.
\end{align*}
On each interval $(x_1^i,y_1^i)$ we bound from below $E_\e$  as in \eqref{stima-l=1}, so that by summing  one obtains
$$
	\sum_{i=1}^N\int_{x_1^i}^{y_1^i}\left[\frac{\tilde{Q}(\e^2u_x)}{\e^3}+\frac{F(u)}{\e}\right]\,dx\geq Nc_\e-C\varepsilon^\alpha,
$$
that is \eqref{E_p(w)>Nc0-eps^l} with $j=1$. 
Arguing inductively as done in the case $N=1$, we obtain \eqref{E_p(w)>Nc0-eps^l} for $E_\e$ in the general case $j\geq2$.
\end{proof}

\subsubsection{Some comments on the lower bounds}
First of all, let us notice that in the proof of Proposition \ref{prop:lower_deg}, we simply took advantage of the behavior of $Q$ in a neighborhood of zero, 
and of the fact that the potential $F \in C(\R)$ satisfies $F(u)>0$ for any $u \neq \pm 1$ and \eqref{eq:ass-F3} for some $\theta>0$. 
Hence, Proposition \ref{prop:lower_deg} holds also in the case $\theta\in(0,2]$; 
moreover, since in such a case $\alpha>1$, the increasing sequence in \eqref{eq:exp_alg} is unbounded, and we can rewrite the estimate \eqref{E_p(w)>Nc0-eps^l} as
\begin{equation*}
	E_\e[u]\geq Nc_\e-C_k\varepsilon^{k}, \qquad \qquad k\in\mathbb{N},
\end{equation*}
provided that
\begin{equation*}
	E_\e[u]\leq Nc_\e+C\varepsilon^{k},\qquad \qquad k\in\mathbb{N}.
\end{equation*}
Nevertheless, if $\theta=2$, we underline that \eqref{eq:lower} provides a {\it stronger} lower bound, where the error is exponentially small rather than algebraically small. 

On the other hand, if $\theta>2$ then $\alpha\in(0,1)$, and, consequently
\begin{equation}\label{eq:beta}
	\lim_{j\to+\infty}k_j= \sum_{m=1}^{+\infty} \alpha^m =\frac{1}{1-\alpha}-1=\frac{\theta+2}{\theta-2}:=\beta.
\end{equation}
Therefore, in the case $\theta>2$ one only has the lower bounds 
\begin{equation*}
	E_\e[u]\geq Nc_\e-C_j\varepsilon^{k_j}, \qquad \mbox{for any } j\in\mathbb N,\qquad \mbox{ with } \qquad \lim_{j\to+\infty}k_j=\beta.
\end{equation*}

As a direct consequence of Propositions \ref{prop:lower}-\ref{prop:lower_deg} we have the following result, showing that if the family $\{ u_\e\}_{\e >0}$ makes $N$ transitions among $+1$ and $-1$ in an {\it energy efficient way} (see \eqref{TLS2} for the rigorous definition), then $E_\e[u^\e]$ converges, as $\e \to 0$, to the minimum energy in the case of the classical Ginzburg-Landau functional (see \cite{Bron-Kohn}).

\begin{cor}\label{cor:TLS}
Let $v$ as in \eqref{vstruct} and let $u^\e \in H^1(a,b)$ be such that 
\begin{equation}\label{TLS1}
	\lim_{\varepsilon\rightarrow 0^+} \| u^\e -v \|_{L^1}=0,
\end{equation}
and there exists a function $\nu \, : \, (0,1) \to (0,1)$ such that
\begin{equation}\label{TLS2}
	E_\e[u^\e] \leq N c_\e +\nu(\e), \qquad \mbox{with} \qquad \lim_{\e \to 0^+} \nu(\e) =0,
\end{equation}
where $c_\e$ is defined in \eqref{eq:c_eps}. Then
\begin{equation*}
\lim_{\e \to 0^+} E_\e[u^\e] = N c_0, \qquad \mbox{ with } \qquad c_0:= \lim_{\e\to0^+} c_\e=\sqrt{Q'(0)}\int_{-1}^{+1} \sqrt{2F(s)}\,ds.
\end{equation*}
\end{cor}
\begin{proof}
By Propositions \ref{prop:lower}-\ref{prop:lower_deg} and from \eqref{TLS2} we have:
\begin{equation}\label{eq:proofTLS}
Nc_\e -C R_{\theta, \e} \leq E_\e[u^\e] \leq Nc_\e + \nu(\e),
\end{equation}
where 
\begin{equation}\label{def:Rte}
R_{\theta, \e}= \left\{ \begin{aligned}
& \exp(-A/\varepsilon) \quad  &\mbox{if}  \quad &\theta=2 \\
& \e^{k_j}  \quad  &\mbox{if}  \quad &\theta>2,
\end{aligned}\right.
\end{equation}
with  $A$ appearing in Proposition \ref{prop:lower} and $\{ k_j \}_{j \in \N}$ defined in \eqref{eq:exp_alg}. 
Moreover, by using \eqref{eq:J-important} and \eqref{eq:expQ}, we deduce that
\begin{align*}
	Q(\e^2J_\e(F(s)))&=Q\left(\sqrt{\frac{2\e^2}{Q'(0)}F(s)}+\rho(\e^2F(s))\right)\\
	&=Q'(0)\left[\sqrt{\frac{2\e^2}{Q'(0)}F(s)}+\rho(\e^2F(s))\right]+o(\e),
\end{align*}
and substituting into the definition \eqref{eq:c_eps}, we end up with
\begin{equation*}
	\lim_{\e\to0^+} c_\e=\e^{-1}\int_{-1}^1 Q(\e^2J_\e(F(s)))\,ds
	=\sqrt{Q'(0)}\int_{-1}^{+1} \sqrt{2F(s)}\,ds.
\end{equation*}
Hence, the thesis follows by simply passing to the limit as $\e \to 0^+$ in \eqref{eq:proofTLS}.
\end{proof}

\subsection*{Example of a function satisfying the assumptions of Corollary \ref{cor:TLS}} 
We conclude this section by showing that there exist a family of functions satisfying assumptions \eqref{TLS1}-\eqref{TLS2}. 
First of all, we observe that it is easy to check that the {\it compactons} $\varphi_1$ and $\varphi_2$ constructed in Proposition \ref{prop:comp} satisfy \eqref{TLS1} 
and $E_\e[\varphi_{1}] =E_\e[\varphi_{2}] = N c_\e$.
However, such stationary solutions exist only in the case $\theta\in(1,2)$; 
the idea is to use a similar construction as the one done for compactons to obtain a function satisfying \eqref{TLS1}-\eqref{TLS2} also when $\theta \geq 2$; 
we underline that in this case these profiles are not stationary solutions.

Let us thus consider the increasing standing wave $\Phi_\e$,  solution to \eqref{eq:Fi}, and observe that 
\begin{equation*}
	\lim_{\e\to0}\Phi_\e(x)=
		\begin{cases}
			-1, \qquad & x<0,\\	
			0, &x=0,\\
			+1, & x>0.
		\end{cases}
\end{equation*}
Now, choose $\e_0>0$ small enough so that the condition \eqref{eq:maxF} holds true for any $\e\in(0,\e_0)$ and
fix $N\in\mathbb{N}$  transition points $a<h_1<h_2<\dots<h_n<b$. Denoted  by
\begin{equation*}
	m_1:=a, \qquad \quad m_j:=\frac{h_{j-1}+h_j}{2}, \quad j=2,\dots,N-1, \qquad \quad m_N:=b,
\end{equation*}
the middle points, we define
\begin{equation}\label{eq:translayer}
	u^\e(x):=\Phi_\e\left((-1)^j(x-h_j)\right), \qquad \qquad x\in[m_j,m_{j+1}], \qquad \qquad j=1,\dots N.
\end{equation}
Notice that $u^\e(h_j)=0$, for $j=1,\dots,N$ and for definiteness we choose $u^\e(a)<0$ (the case $u^\e(a)>0$ is analogous).
Let us now prove that $u^\e$ satisfies \eqref{TLS1}-\eqref{TLS2}.
It is easy to check that $u^\e\in H^1(a,b)$ and satisfies \eqref{TLS1}; concerning \eqref{TLS2},
the definitions of $E_\e$ and $u^\e$ give
\begin{equation*}
	E_\e[u^\e]=\sum_{j=1}^{N}\int_{m_j}^{m_{j+1}}\left[\frac{\tilde Q(\e^2\Phi'_\e)}{\e^3}+\frac{F(\Phi_\e)}\e\right]\,dx.
\end{equation*}
From \eqref{eq:impforenergy}, it follows that
\begin{align*}
	\int_{m_j}^{m_{j+1}}\left[\frac{\tilde Q(\e^2\Phi'_\e)}{\e^3}+\frac{F(\Phi_\e)}\e\right]\,dx&=
	\e^{-1}\int_{m_j}^{m_{j+1}}\left[\Phi'Q(\e^2 J_\e (F(\Phi_\e)))\right]\,dx\\
	&=\e^{-1}\int_{\Phi_\e(m_j)}^{\Phi_\e(m_{j+1})}\left[Q(\e^2 J_\e (F(s)))\right]\,ds<c_\e,
\end{align*}
where $c_\e$ is defined in \eqref{eq:c_eps}.
Summing up all the terms we end up with $E_\e[u^\e]\leq Nc_\e$ which clearly implies \eqref{TLS2}.

\section{Slow motion}\label{sec:slow}	
In this last section we investigate the long time dynamics of the solutions to \eqref{eq:Q-model}-\eqref{eq:Neu}-\eqref{eq:initial} 
with a special focus to their different speed rate of convergence towards an asymptotic configuration, 
which heavily depends on the parameter $\theta$ appearing in the potential $F$ (see the assumptions \eqref{ipoF1} and \eqref{ipoF2}).
Specifically, the main results of this section are contained in Theorems \ref{thm:main} and \ref{thm:main2}:  
in the former we prove that in the critical case $\theta=2$ the solutions to \eqref{eq:Q-model}-\eqref{eq:Neu}-\eqref{eq:initial} exhibit a metastable behaviour, 
namely, they maintain the same unstable structure of the initial datum for a time which is exponentially long with respect to the parameter $\e$; 
in the latter, concerning instead the supercritical case $\theta>2$, we prove slow motion with a speed rate which is only algebraic with respect to $\e$. 

\subsection*{Preliminary assumptions} 
Here and in the rest of the section, we fix a function $v$ as in \eqref{vstruct} and we assume that the initial datum in \eqref{eq:initial} depends on $\e$ and satisfies
\begin{equation}\label{eq:ass-u0}
	\lim_{\varepsilon\rightarrow 0^+} \|u^\varepsilon_0-v\|_{{}_{L^1}}=0.
\end{equation}
Moreover, we assume that there exist  $C, \e_0>0$ such that, for any $\e\in(0,\e_0)$,
\begin{equation}\label{eq:energy-ini}
	E_\varepsilon[u^\varepsilon_0]\leq Nc_\e+ C R_{\theta,\e},
\end{equation}
where $R_{\theta,\e}$ is defined in \eqref{def:Rte}, that is $u^\e_0$ satisfies the assumptions of Corollary \ref{cor:TLS}.

We emphasize that an initial datum  as the one satisfying the assumptions \eqref{eq:ass-u0}-\eqref{eq:energy-ini} is far from being a stationary solution if and only if $\theta \geq 2$; indeed, when $\theta  \in (1,2)$, in Proposition \ref{prop:comp} we proved the existence of a particular class of stationary solutions (compactons), which have an arbitrary number of transition layers that are randomly located inside the interval $[a,b]$. 
Hence, an initial datum satisfying \eqref{eq:ass-u0}-\eqref{eq:energy-ini} is either a steady state or a small perturbation of it; 
thus, proving that the corresponding time dependent solution maintains the same structure for long times is either trivially true 
(indeed, the same structure is maintained for all $t>0$) or only a partial result,
since we do not know if such structure will be lost at some point. 
In other words, in the case $\theta\in(1,2)$ transition layers do not evolve in time and could persist forever.
On the contrary, if $\theta \geq 2$, then stationary solutions can only have layers that are equidistant (see Proposition \ref{periodic:bounded}), 
so that an initial configuration as $u^\e_0$ is in general far away from any steady state.

\subsection{Exponentially slow motion}\label{sec:exp_met}
In this subsection we examine the persistence of layered solutions to \eqref{eq:Q-model}-\eqref{eq:Neu}-\eqref{eq:initial} when $\theta=2$ in the assumption \eqref{ipoF2} on the potential $F$. In particular, we prove that in this case a metastability phenomenon occurs, showing that the solutions perpetuate the same behavior of the initial datum for an $\e$-exponentially long time, that is, at least for a time equals $m\,e ^{A/\e}$ for some $A>0$ and any $m>0$, both independent on $\e$, as stated in the following theorem.

\begin{thm}[exponentially slow motion when $\theta=2$]\label{thm:main}
Assume that $Q\in C^1(\R)$ satisfies \eqref{eq:Q-ass1}-\eqref{eq:Q-ass2} and that $F\in C^1(\R)$ satisfies \eqref{ipoF1}-\eqref{ipoF2} with $\theta=2$. 
Let $v$ be as in \eqref{vstruct} and $A\in(0,r\sqrt{2\lambda_1 \mathcal{Q}^{-1}})$, with $\mathcal{Q}$ defined in \eqref{eq:mathc-Q} and
$\lambda_1>0$ (independent on $\e$) as in \eqref{ipoF2}.
If $u^\varepsilon$ is the solution of \eqref{eq:Q-model}-\eqref{eq:Neu}-\eqref{eq:initial}
with initial datum $u_0^{\varepsilon}$ satisfying \eqref{eq:ass-u0} and \eqref{eq:energy-ini} with $R_{\theta,\e}=  \exp(-A/\e)$, then
\begin{equation}\label{eq:limit}
	\sup_{0\leq t\leq m\exp(A/\varepsilon)}\|u^\varepsilon(\cdot,t)-v\|_{{}_{L^1}}\xrightarrow[\varepsilon\rightarrow0]{}0,
\end{equation}
for any $m>0$.
\end{thm}
The proof of Theorem \ref{thm:main} strongly relies on the following result which provides an estimate from above of the $L^2$-norm of the derivative with respect to time of the solution $u^\e(\cdot,t)$ to \eqref{eq:Q-model}-\eqref{eq:Neu}-\eqref{eq:initial} under the assumptions of Theorem \ref{thm:main}. 
\begin{prop}\label{prop:L2-norm}
Under the same assumptions of Theorem \ref{thm:main}, there exist positive constants $\varepsilon_0, C_1, C_2>0$ (independent on $\varepsilon$) such that
\begin{equation}\label{L2-norm}
	\int_0^{C_1\varepsilon^{-1}\exp(A/\varepsilon)}\|u_t^\varepsilon\|^2_{{}_{L^2}}dt\leq C_2\varepsilon\exp(-A/\varepsilon),
\end{equation}
for all $\varepsilon\in(0,\varepsilon_0)$.
\end{prop}

\begin{proof}
Let $\varepsilon_0>0$ so small that for all $\varepsilon\in(0,\varepsilon_0)$, \eqref{eq:energy-ini} holds and 
\begin{equation}\label{1/2delta}
	\|u_0^\varepsilon-v\|_{{}_{L^1}}\leq\frac12\delta,
\end{equation}
where $\delta$ is the constant of Proposition \ref{prop:lower}. 
Let $\hat T>0$; we claim that if
\begin{equation}\label{claim1}
	\int_0^{\hat T}\|u_t^\varepsilon\|_{{}_{L^1}}dt\leq\frac12\delta,
\end{equation}
then there exists $C>0$ such that
\begin{equation}\label{claim2}
	E_\varepsilon[u^\varepsilon](\hat T)\geq Nc_\e-C\exp(-A/\varepsilon).
\end{equation}
Indeed, inequality \eqref{claim2} follows from Proposition \ref{prop:lower} if $\|u^\varepsilon(\cdot,\hat T)-v\|_{{}_{L^1}}\leq\delta$.
By using triangle inequality, \eqref{1/2delta} and \eqref{claim1}, we obtain
\begin{equation*}
	\|u^\varepsilon(\cdot,\hat T)-v\|_{{}_{L^1}}\leq\|u^\varepsilon(\cdot,\hat T)-u_0^\varepsilon\|_{{}_{L^1}}+\|u_0^\varepsilon-v\|_{{}_{L^1}}
	\leq\int_0^{\hat T}\|u_t^\varepsilon\|_{{}_{L^1}}+\frac12\delta\leq\delta.
\end{equation*}
Substituting \eqref{eq:energy-ini} and \eqref{claim2} in \eqref{eq:energyestimate}, one has 
\begin{equation}\label{L2-norm-Teps}
	\int_0^{\hat T}\|u_t^\varepsilon\|^2_{{}_{L^2}}dt\leq C_2\e\exp(-A/\varepsilon).
\end{equation}
It remains to prove that inequality \eqref{claim1} holds for $\hat T\geq C_1\e^{-1}\exp(A/\varepsilon)$.
If 
\begin{equation*}
	\int_0^{+\infty}\|u_t^\varepsilon\|_{{}_{L^1}}dt\leq\frac12\delta,
\end{equation*}
there is nothing to prove. 
Otherwise, choose $\hat T$ such that
\begin{equation*}
	\int_0^{\hat T}\|u_t^\varepsilon\|_{{}_{L^1}}dt=\frac12\delta.
\end{equation*}
Using  H\"older's inequality and \eqref{L2-norm-Teps}, we infer
\begin{equation*}
	\frac12\delta\leq[\hat T(b-a)]^{1/2}\biggl(\int_0^{\hat T}\|u_t^\varepsilon\|^2_{{}_{L^2}}dt\biggr)^{1/2}\leq
	\left[\hat T(b-a)C_2\varepsilon\exp(-A/\varepsilon)\right]^{1/2},
\end{equation*}
so that there exists $C_1>0$ such that
\begin{equation*}
	\hat T\geq C_1\varepsilon^{-1}\exp(A/\varepsilon),
\end{equation*}
and the proof is complete.
\end{proof}

Now we are ready to prove Theorem \ref{thm:main}.
\begin{proof}[Proof of Theorem \ref{thm:main}]
Fix $m>0$.
As a consequence of the triangle inequality we have that
\begin{equation}\label{trianglebar}
	\sup_{0\leq t\leq m\exp(A/\varepsilon)}\|u^\varepsilon(\cdot,t)-v\|_{{}_{L^1}}\leq \sup_{0\leq t\leq m\exp(A/\varepsilon)}\|u^\varepsilon(\cdot,t)-u_0^\varepsilon\|_{{}_{L^1}}+\|u_0^\varepsilon-v\|_{{}_{L^1}}.
\end{equation}
Hence, since the second term in the right-hand side of  \eqref{trianglebar} tends to $0$ by \eqref{eq:ass-u0}, in order to prove \eqref{eq:limit}, it is sufficient to show that 
\begin{equation}\label{primopezzo}
	\sup_{0\leq t\leq m\exp(A/\varepsilon)}\|u^\varepsilon(\cdot,t)-u_0^\e\|_{{}_{L^1}}\xrightarrow[\varepsilon\rightarrow0]{}0.
\end{equation}
To this aim, we first observe that up to taking $\e$ so small that $m<C_1\e^{-1}$, we can apply  \eqref{L2-norm} to deduce
\begin{equation}\label{proof:usata}
\int_0^{m\exp(A/\varepsilon)}\|u_t^\varepsilon\|^2_{{}_{L^2}}dt\leq C_2\varepsilon\exp(-A/\varepsilon).
\end{equation}
Moreover, for all $t\in[0,m\exp(A/\varepsilon)]$ we have
\begin{equation*}
\begin{aligned}
\|u^\e(\cdot,t)-u^\e_0\|_{{}_{L^1}}&\leq\int_0^{m\exp(A/\varepsilon)}\|u_t^\e(\cdot,t)\|_{{}_{L^1}}\,dt  \\
&\leq \sqrt{m(b-a)} \exp(A/2\varepsilon) \left( \int_0^{m\exp(A/\varepsilon)}\|u_t^\e(\cdot,t)\|^2_{{}_{L^2}}\,dt \right)^{\frac{1}{2}}
\leq C\sqrt\e,
\end{aligned}
\end{equation*}	
where we applied  H\"older's inequality and \eqref{proof:usata}. 
We thus obtained \eqref{primopezzo} and the proof is complete.
\end{proof}

\begin{rem}\label{rem:Q}
We stress that the constant $\mathcal Q$ defined in \eqref{eq:mathc-Q} (and appearing the first time in the constant $A$ of the lower bound \eqref{eq:lower}) plays a relevant role in the dynamics of the solution; indeed, such lower bound is needed to prove Theorem \ref{thm:main}, and, from estimate \eqref{eq:limit}, we can clearly see that the bigger is $A$ (that is the smaller is $\mathcal Q$) the slower is the dynamics. As to give a hint on what happens even with a small variation of $\mathcal Q$, if we choose $Q'$ so that its maximum is $4$ instead of $1$ (as it is for the examples \eqref{ex:fluxfunction} we considered)  then the time taken for the solution to drift apart from the initial datum $u_0$ reduces from $T_\e$ to $\sqrt{T_\e}$; we will see further details with the numerical simulations of Section \ref{numerics}.
\end{rem}

\subsection{Algebraic slow motion}
In this subsection we consider the case in which the potential $F$ satisfies assumption \eqref{ipoF2} with $\theta>2$
and we show that 
the evolution of the solutions drastically changes with respect to the critical case $\theta=2$, studied in Section \ref{sec:exp_met}.
Indeed, the exponentially slow motion proved in Theorem \ref{thm:main} is a peculiar phenomenon of \emph{non-degenerate potentials}, 
while if $\theta>2$, then the solution maintains the same unstable structure of the initial profile \emph{only} for an algebraically long time with respect to $\e$, 
that is, at least for a time equals $ l\e^{-\beta}$, for any $l>0$, with $\beta>0$ defined in \eqref{eq:beta}.
This is a consequence of the fact that when $\theta>2$, we have no longer a lower bound like the one exhibited in Proposition \ref{prop:lower} (with an exponentially small reminder),
but only a lower bound with an algebraic small reminder, see Proposition \ref{prop:lower_deg}.
Our second main result is the following one.

\begin{thm}[algebraic slow motion when $\theta>2$]\label{thm:main2}
Assume that $Q\in C^1(\R)$ satisfies \eqref{eq:Q-ass1}-\eqref{eq:Q-ass2} and that $F\in C^1(\R)$ satisfies \eqref{ipoF1}-\eqref{ipoF2} with $\theta>2$. 
Moreover, let $v$ be as in \eqref{vstruct} and let $\left\{k_j\right\}_{j\in\mathbb N}$ be as in \eqref{eq:exp_alg}. 
If $u^\varepsilon$ is the solution to \eqref{eq:Q-model}-\eqref{eq:Neu}-\eqref{eq:initial},
with initial profile $u_0^{\varepsilon}$ satisfying \eqref{eq:ass-u0} and \eqref{eq:energy-ini} with $R_{\theta,\e}=\e^{k_j}$, then 
\begin{equation}\label{eq:limit-deg}
	\sup_{0\leq t\leq l{\e^{-k_j}}}\|u^\varepsilon(\cdot,t)-v\|_{{}_{L^1}}\xrightarrow[\varepsilon\rightarrow0]{}0,
\end{equation}
for any $l>0$.
\end{thm}

\begin{proof}
The proof follows the same steps of the proof of Theorem \ref{thm:main} 
and it is obtained by using Proposition \ref{prop:lower_deg} instead of Proposition \ref{prop:lower}.
In particular, proceeding as in the proof of Proposition \ref{prop:L2-norm}, one can prove that
there exist $\varepsilon_0, C_1, C_2>0$ (independent on $\varepsilon$) such that
\begin{equation*}
	\int_0^{C_1\varepsilon^{-(k_j+1)}}\|u_t^\varepsilon\|^2_{{}_{L^2}}dt\leq C_2\varepsilon^{k_j+1},
\end{equation*}
for all $\varepsilon\in(0,\varepsilon_0)$.
Thanks to the latter estimate, we can prove \eqref{eq:limit-deg} in the same way we proved \eqref{eq:limit} (see \eqref{trianglebar} and the following discussion).
\end{proof}

\subsection{Layer Dynamics}
In this last subsection, we ultimate our investigation giving a description of the slow motion 
of the transition points $h_1,\ldots,h_N$.
More precisely, we will incorporate the analysis in both the critical case (that is when $\theta=2$) and the subcritical one (namely $\theta>2$), showing that the transition layers evolve 
with a velocity which goes to zero as $\e\to0^+$, according to Theorem \ref{thm:interface}.  

For this, let us fix a function $v$ as in \eqref{vstruct} and define its {\it interface} $I[v]$ as the set
\begin{equation*}
	I[v]:=\{h_1,h_2,\ldots,h_N\}.
\end{equation*}
Moreover, for any function $u:[a,b]\rightarrow\mathbb{R}$ and for any closed subset $K\subset\R\backslash\{\pm1\}$,
the {\it interface} $I_K[u]$ is defined by
\begin{equation*}
	I_K[u]:=u^{-1}(K).
\end{equation*}
Finally, we recall the notion of {\it Hausdorff distance}  between any two subsets $A$ and $B$ of $\mathbb{R}$ denoted with $d(A,B)$ and given by 
\begin{equation*}
	d(A,B):=\max\biggl\{\sup_{\alpha\in A}d(\alpha,B),\,\sup_{\beta\in B}d(\beta,A)\biggr\},
\end{equation*}
where $d(\beta,A):=\inf\{|\beta-\alpha|: \alpha\in A\}$, for every $\beta\in B $.
 
Before stating the main result of this subsection (see Theorem \ref{thm:interface}), we prove the following lemma which is merely variational, meaning that it does not take into account the equation \eqref{eq:Q-model} , and establishes that, if a function $u\in H^1([a,b])$ is close to $v$ in $L^1$ and its energy
$E_\varepsilon[u]$ (defined in \eqref{eq:energy}) exceeds for a small quantity with respect to $\e$ the minimum energy to have $N$ transitions, then 
the distance between the interfaces $I_K[u]$ and $I_K[v]$ remains small.   
\begin{lem}\label{lem:interface}
Assume that $Q\in C^1(\R)$ satisfies \eqref{eq:Q-ass1}-\eqref{eq:Q-ass2}, $F\in C^1(\R)$ satisfies \eqref{ipoF1}-\eqref{ipoF2} with $\theta\geq 2$ 
and let $v$ be as in \eqref{vstruct}.
Given $\delta_1\in(0,r)$ and a closed subset $K\subset\R\backslash\{\pm1\}$, 
there exist positive constants $\hat\delta,\varepsilon_0$  such that, if for all $\varepsilon\in(0,\varepsilon_0)$  $u\in H^1([a,b])$ satisfies
\begin{equation}\label{eq:lem-interf}
	\|u-v\|_{{}_{L^1}}<\hat\delta \qquad \quad \mbox{ and } \qquad \quad E_\varepsilon[u]\leq Nc_\e+M_\e,
\end{equation}
for some $M_\e>0$ and with $E_\e[u]$ defined in \eqref{eq:energy}, we have
\begin{equation}\label{lem:d-interfaces}
	d(I_K[u], I[v])<\tfrac12\delta_1.
\end{equation}
\end{lem}
\begin{proof}
Fix $\delta_1\in(0,r)$ and choose $\rho>0$ small enough that 
\begin{equation*}
	I_\rho:=(-1-\rho,-1+\rho)\cup(1-\rho,1+\rho)\subset\R\backslash K, 
\end{equation*}
and 
\begin{equation*}
	\inf\left\{\e^{-1}\left|\int_{\xi_1}^{\xi_2}Q\left(\e^2 J_\e(F(s))\right)\,ds\right| : \xi_1\in K, \xi_2\in I_\rho\right\}>2M_\e,
\end{equation*}
where
\begin{equation*}
	M_\e:=2N\e^{-1}\max\left\{\int_{1-\rho}^{1}Q\left(\e^2 J_\e(F(s))\right)\,ds, \, \int_{-1}^{-1+\rho}Q\left(\e^2 J_\e(F(s))\right)\,ds \right\}.
\end{equation*}
By using the first assumption in \eqref{eq:lem-interf} and by reasoning as in the proof of \eqref{2points} in Proposition \ref{prop:lower}, 
we can prove that, if we consider $h_i$ the discontinuous points of $v$, then for each $i= 1, \dots , N$ there exist
\begin{equation*}
	x^-_{i}\in(h_i-\delta_1/2,h_i) \qquad \textrm{and} \qquad x^+_{i}\in(h_i,h_i+\delta_1/2),
\end{equation*}
such that
\begin{equation*}
	|u(x^-_{i})-v(x^-_{i})|<\rho \qquad \textrm{and} \qquad |u(x^+_{i})-v(x^+_{i})|<\rho.
\end{equation*}
Now suppose by contraddicition that \eqref{lem:d-interfaces} is violated. 
Using \eqref{eq:strangeineq+}, we deduce
\begin{align}
	E_\varepsilon[u]\geq&\sum_{i=1}^N\left|\e^{-1}\int_{u(x^-_{i})}^{u(x^+_{i})}Q\left(\e^2 J_\e(F(s))\right)\,ds\right|\notag\\ 
	& \qquad +\inf\left\{\left|\e^{-1}\int_{\xi_1}^{\xi_2}Q\left(\e^2 J_\e(F(s))\right)\,ds\right| : \xi_1\in K, \xi_2\in I_\rho\right\}. \label{diseq:E1}
\end{align}
On the other hand, we have
\begin{align*}
	\left|\e^{-1}\int_{u(x^-_{i})}^{u(x^+_{i})}Q\left(\e^2 J_\e(F(s))\right)\,ds\right|&\geq\e^{-1}\int_{-1}^{1}Q\left(\e^2 J_\e(F(s))\right)\,ds\\
	&\qquad-\e^{-1}\int_{-1}^{-1+\rho}Q\left(\e^2 J_\e(F(s))\right)\,ds\\
	&\qquad -\e^{-1}\int_{1-\rho}^{1}Q\left(\e^2 J_\e(F(s))\right)\,ds\\
	&\geq c_\e-\frac{M_\e}{N}. 
\end{align*}
Substituting the latter bound in \eqref{diseq:E1}, we deduce
\begin{equation*}
	E_\varepsilon[u]\geq Nc_\e-M_\e+\inf\left\{\left|\e^{-1}\int_{\xi_1}^{\xi_2}Q\left(\e^2 J_\e(F(s))\right)\,ds\right| : \xi_1\in K, \xi_2\in I_\rho\right\},
\end{equation*}
which implies, because of the choice of $\rho$,  that
\begin{align*}
	E_\varepsilon[u]>Nc_\e+M_\e,
\end{align*}
which is a contradiction with assumption \eqref{eq:lem-interf}. Hence, the bound \eqref{lem:d-interfaces} is true and the proof is completed.
\end{proof}

Finally, thanks to Theorems \ref{thm:main}, \ref{thm:main2} and Lemma \ref{lem:interface}, we can prove the main result of the present subsection, which provides information about the slow motion of the transition layers as $\e$ goes to 0. In particular, we highlight that the minimum time so that the distance between the interface of the time-dependent solution $u^\e(\cdot,t)$ and that one of the initial datum becomes greater than a fixed quantity is exponentially big with respect to $\e$ in the critical case  $\theta=2$, while only algebraically large in the super critical case $\theta>2$. 
\begin{thm}\label{thm:interface}
Assume that $Q\in C^1(\R)$ satisfies \eqref{eq:Q-ass1}-\eqref{eq:Q-ass2} and that $F\in C^1(\R)$ satisfies \eqref{ipoF1}-\eqref{ipoF2} with {$\theta\geq 2$}.
Let $u^\varepsilon$ be the solution of \eqref{eq:Q-model}-\eqref{eq:Neu}-\eqref{eq:initial}, 
with initial datum $u_0^{\varepsilon}$ satisfying \eqref{eq:ass-u0} and \eqref{eq:energy-ini}. 
Given $\delta_1\in(0,r)$ and a closed subset $K\subset\R\backslash\{\pm1\}$, set
\begin{equation*}
	t_\varepsilon(\delta_1)=\inf\{t:\; d(I_K[u^\varepsilon(\cdot,t)],I_K[u_0^\varepsilon])>\delta_1\}.
\end{equation*}
Then, there exists $\varepsilon_0>0$ such that if $\varepsilon\in(0,\varepsilon_0)$
\begin{equation*}
	t_\varepsilon(\delta_1)> \left\{ \begin{aligned}
& \exp(A/\varepsilon) \quad  &\mbox{if}  \quad &\theta=2, \\
& \e^{-k_j}   \quad  &\mbox{if}  \quad &\theta>2,
\end{aligned}\right.
\end{equation*}
where $A$ and $k_j$ are defined as in Propositions \ref{prop:lower} and \ref{prop:lower_deg}, respectively. 
\end{thm}
\begin{proof}
First of all we notice  that if $u_0^\e$ satisfies \eqref{eq:ass-u0} and \eqref{eq:energy-ini}, then there exists  $\varepsilon_0>0$ so small that  $u_0^\e $ automatically verifies assumption \eqref{eq:lem-interf} for all $\varepsilon\in(0,\varepsilon_0)$. 
Hence, we are in the position to apply Lemma \ref{lem:interface}, obtaining that for all $\varepsilon\in(0,\varepsilon_0)$
\begin{equation}\label{interfaces-u0}
	d(I_K[u_0^\varepsilon], I[v])<\tfrac12\delta_1.
\end{equation}
Now, for each fixed $\varepsilon\in(0,\varepsilon_0)$, we consider $u^\varepsilon(\cdot,t)$ for all time $t>0$ such that
\begin{equation}\label{time}
t\leq\left\{ \begin{aligned}
& \exp(A/\varepsilon) \quad  &\mbox{if}  \quad &\theta=2, \\
& \e^{-k_j}   \quad  &\mbox{if}  \quad &\theta>2.
\end{aligned}\right.
\end{equation}
Then, $u^\varepsilon(\cdot,t)$ satisfies the first condition in assumption \eqref{eq:lem-interf} either thanks to \eqref{eq:limit} if $\theta=2$ or by \eqref{eq:limit-deg} if $\theta>2$. The second condition in \eqref{eq:lem-interf} can be easily deduced observing that the energy $E_\varepsilon[u^\varepsilon](t)$ is a non-increasing function of $t$.
Then, we have also that
\begin{equation}\label{interfaces-u}
	d(I_K[u^\varepsilon(\cdot,t)], I[v])<\tfrac12\delta_1
\end{equation}
for all $t$ as in \eqref{time}.  Combining \eqref{interfaces-u0} and \eqref{interfaces-u}, and using the triangle inequality, we deduce that
\begin{equation*}
	d(I_K[u^\varepsilon(\cdot,t)],I_K[u_0^\varepsilon])<\delta_1
\end{equation*}
for all $t$ as in \eqref{time}, as desired.
\end{proof}

\begin{rem}
We underline, according to Theorem \ref{thm:interface}, one must wait an extremely long time (which is either exponentially or algebraically long, depending on wether $\theta=2$ or $\theta >2$) to see an appreciable change in the position of the zeros of $u^\e$. Once again, this proves a slow dynamics of the solution only because 
$\theta \geq 2$; indeed, in such a case we are sure that an initial datum $u_0^\e$ satisfying  \eqref{eq:ass-u0}-\eqref{eq:energy-ini} is neither a stationary solution nor  is close to it, implying that there exists a {\it finite} time $\bar t$ such that the solution $u^\e(\cdot, \bar t)$  will drift apart from $u_0^\e$. Hence, proving that the layers of $u^\e(\cdot,  t)$ stay close to the layers of $u_0^\e$  for  long times is not a trivial result.
\end{rem}

\subsection{Numerical experiments}\label{numerics}
We conclude the paper with some numerical simulations showing the slow evolution of the solutions to \eqref{eq:Q-model}-\eqref{eq:Neu}-\eqref{eq:initial} rigorously described in the previous analysis. All the numerical computations, done for the sole purpose of illustrating the theoretical results, were performed using the built-in solver \texttt{pdepe} by \textsc{Matlab}$^\copyright$, which is a set of tools to solve PDEs in one space dimension.
In all the examples we consider equation \eqref{eq:Q-model} with $Q$ given by one of the two explicit functions of \eqref{ex:fluxfunction}, while $F(u)=\frac{1}{2\theta}|1-u^2|^\theta$ (see \eqref{F:ex}) with different values of $\theta \geq 2$, depending on whether we aim at showing exponentially or algebraic slow motion.

\begin{figure}[t]
	\begin{center}
		\includegraphics[width=7cm,height=5.7cm]{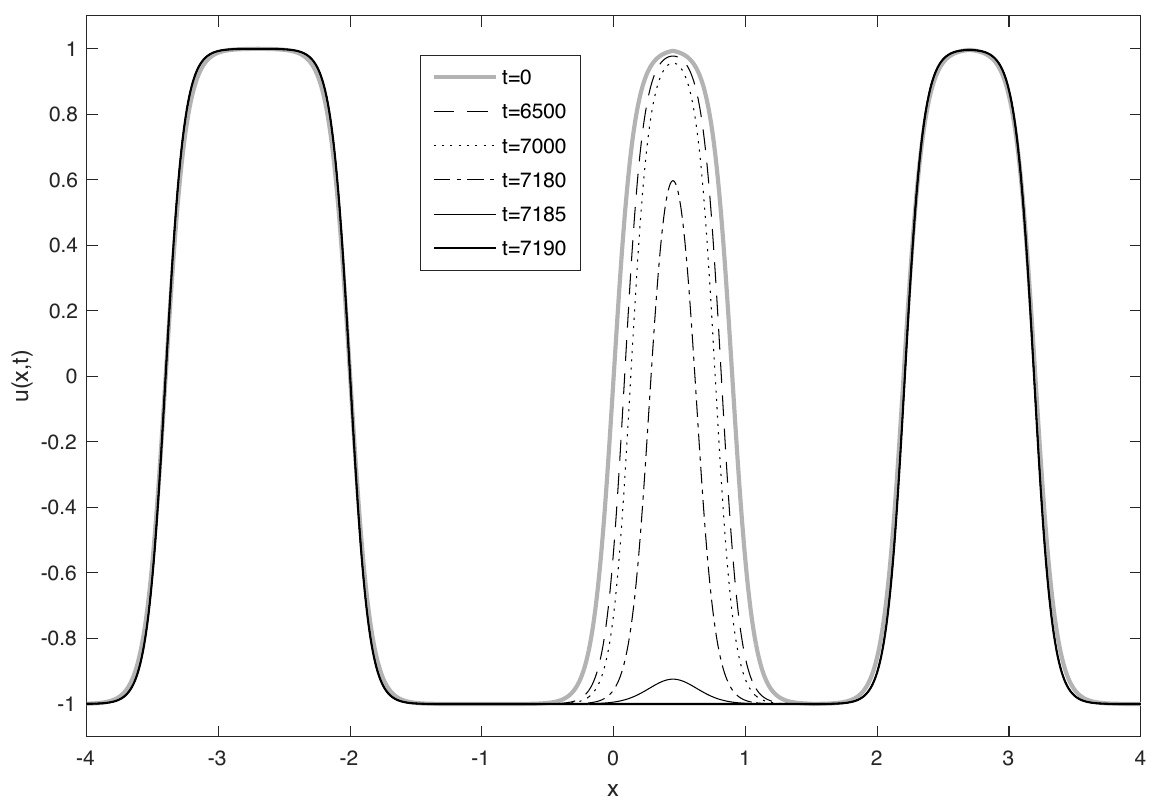}
		\,
		\includegraphics[width=7cm,height=5.7cm]{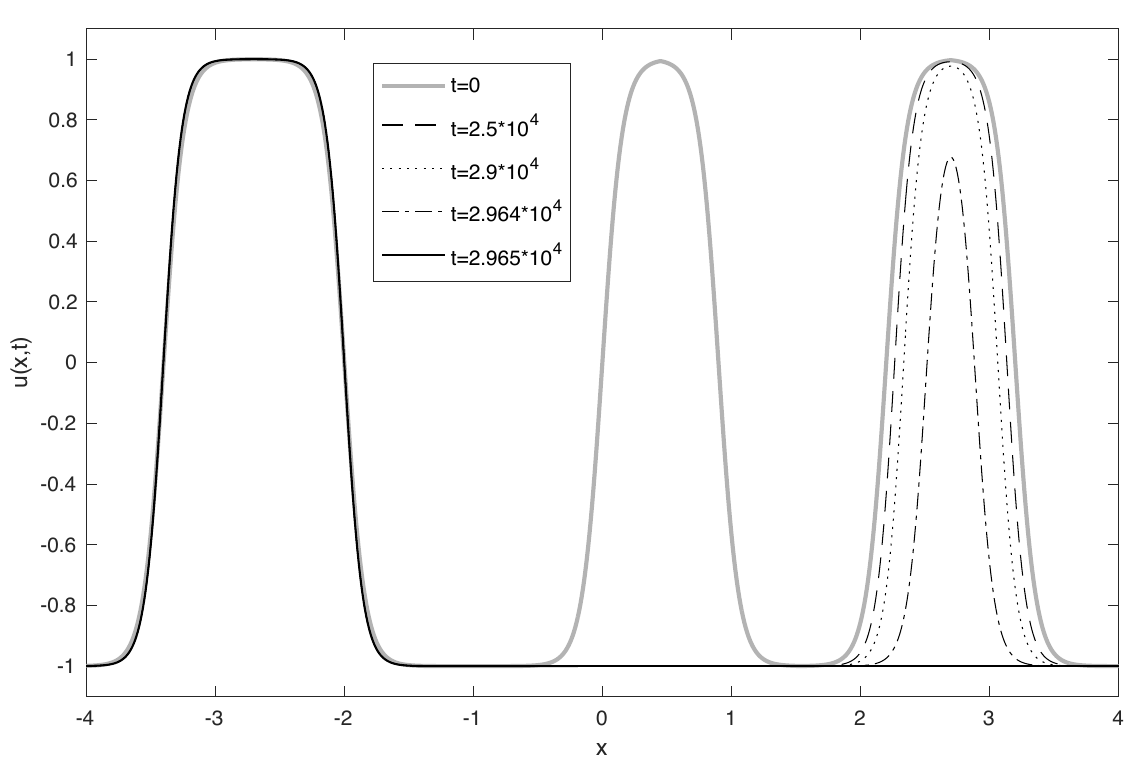}
		\hspace{3mm}
		\caption{\small{In this figure we depict the solution to \eqref{eq:Q-model}-\eqref{eq:Neu}-\eqref{eq:initial}, in the case $Q(s)=\frac{s}{1+s^2}$, $\e=0.1$ and $F(u)=\frac{1}{4}(1-u^2)^2$}. The initial datum $u_0$ has six layers located, respectively, at $-3.4$, $-2$, $0$, $0.9$, $2.2$ and $3.2$.}
		\label{Num1}
	\end{center}
\end{figure}

\subsubsection{Example no. 1} 
We start with an example illustrating the result of Theorem \ref{thm:main}; indeed, in this case $F(u)=\frac{1}{4}(1-u^2)^2$ (hence we are in the critical case $\theta=2$), and we choose $Q(s)=\frac{s}{1+s^2}$. In the left picture of Figure \ref{Num1}  we can see the solution maintains six transitions until $t=7*10^3$ and, suddenly, the closest layers collapse; after that, one has to wait up to $t=2.9*10^4$ to see another appreciable change in the solution, see the right hand picture of Figure \ref{Num1}.
It is worth mentioning that the distance between the first layers which disappear (left hand picture) is $d=0.9$, while the distance of the layers disappearing in the right hand picture is $d=1$; hence, a \emph{small} variation in the distance between the layers at the initial time $t=0$ gives rise to a \emph{big} change on the time taken for the solution to annihilate them.

\begin{figure}[h]
	\begin{center}
		\includegraphics[width=7.3cm,height=5.7cm]{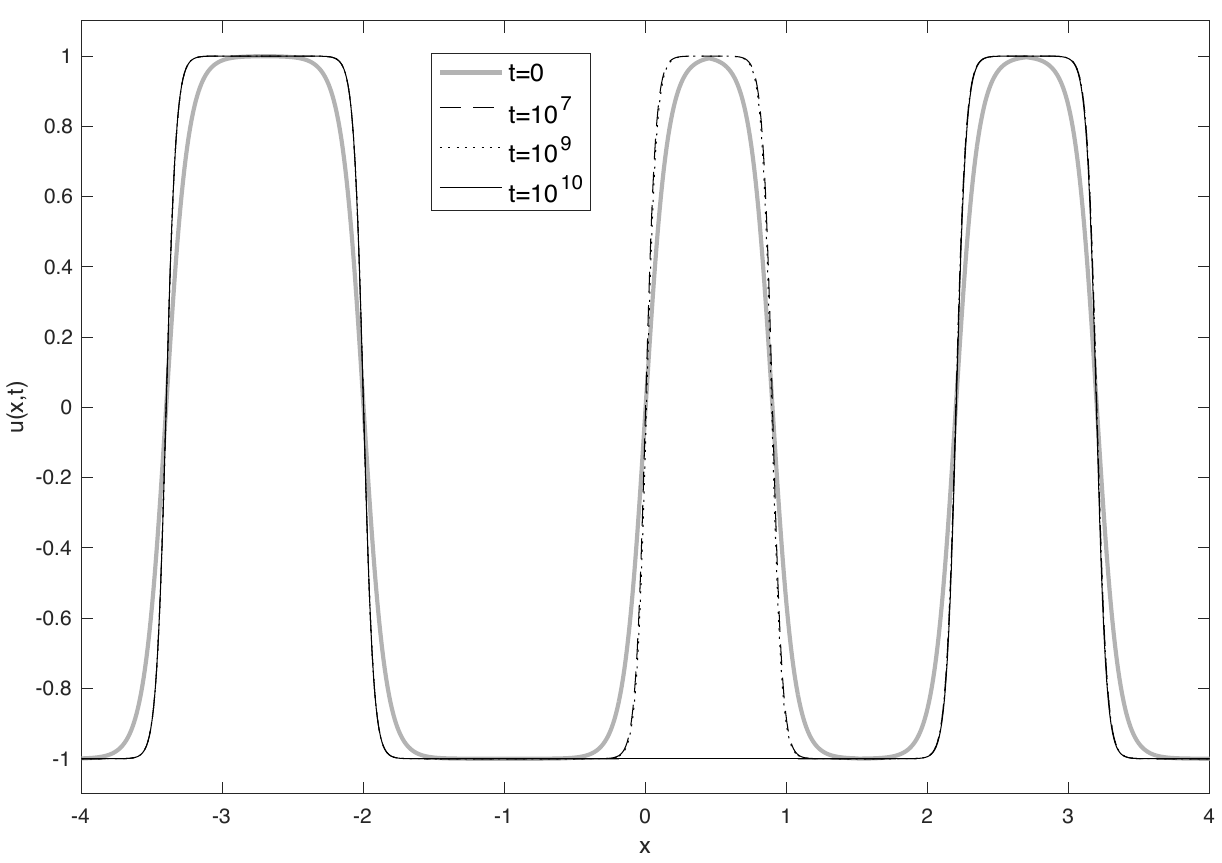}
		\,
		\includegraphics[width=7cm,height=5.7cm]{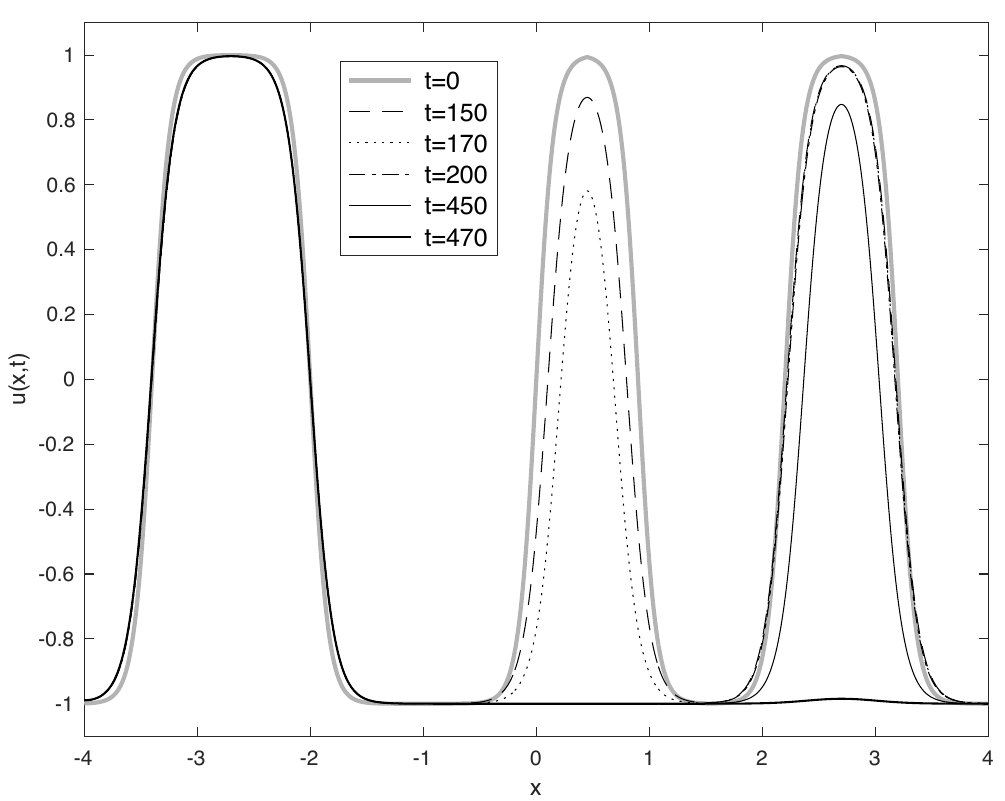}
		\hspace{3mm}
		\caption{In this figure we consider the same problem as in Figure \ref{Num1}, with the only difference $Q(s)= \frac{\alpha s}{1+s^2}$, $\alpha>0$. 
			In the left hand side $\alpha=1/4$, so that $\mathcal Q=1/4$, while on the right hand side $\alpha=\mathcal Q=2$.}
		\label{Num2}
	\end{center}
\end{figure}

\subsubsection{Example no. 2} In this second numerical experiment, we emphasize which is the role of the constant $\mathcal Q=\max \left\{  Q'(s) \, : \, s \in [-\kappa, \kappa] \right\}$ in the metastable dynamics of the solutions; in particular, and as we noticed in Remark \ref{rem:Q}, the constant $\mathcal Q$ affects the evolution of the solution, since it appears in the minimum time needed for the solution to {drift apart} from its initial transition layer structure (see Theorem \ref{thm:main}).
To be more precise, if $\mathcal Q=1$ as in the example of Figure \ref{Num1}, the first {\it bump} collapses at $t \approx 7190$, while the second at $t \approx 3*10^4$; such result has  to be compared with Figure \ref{Num2}, where all the data are the same as Figure \ref{Num1} except for the choice of $\mathcal Q$. 
In the left hand side of Figure \ref{Num2} we choose $Q$ such that  $\mathcal Q=1/4$, and we can see that the evolution becomes much slower, as the first bump collapses at $t=10^{10}$, instead of $t=7190$; on the contrary, if  $\mathcal Q=2$ (right hand picture), the dynamics accelerates and the first two bumps disappear respectively at $t\approx200$ and  $t\approx470$.

\begin{figure}[t]
	\begin{center}
		\includegraphics[width=7.5cm,height=5.7cm]{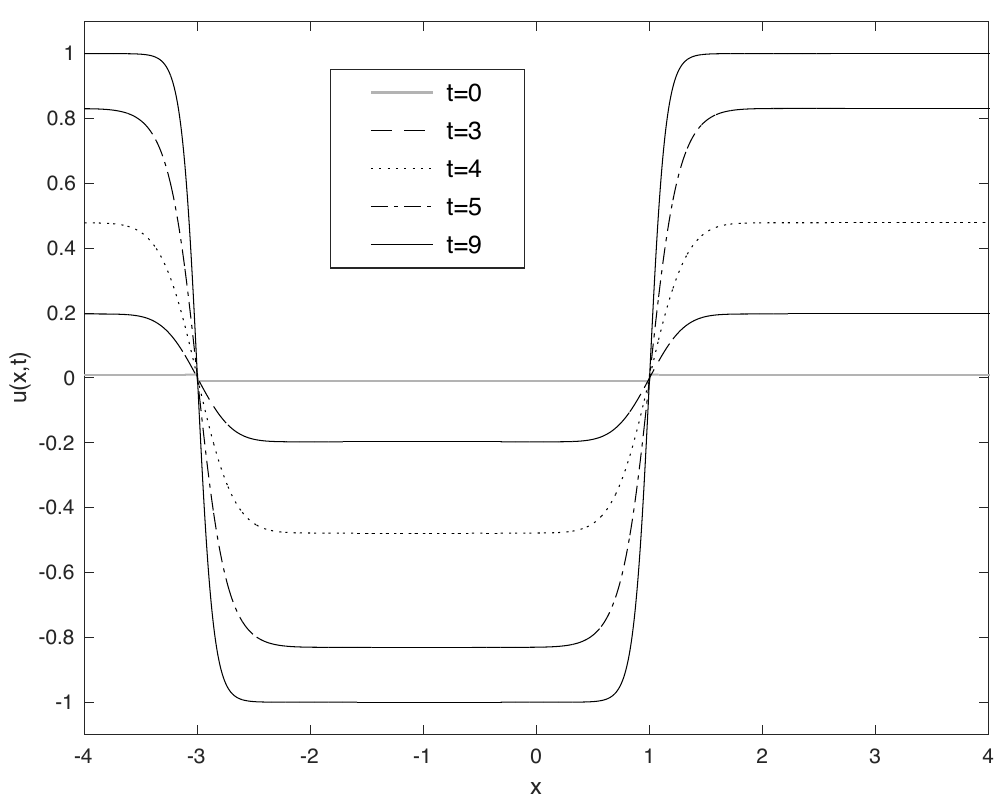}
		\,
		\includegraphics[width=7cm,height=5.7cm]{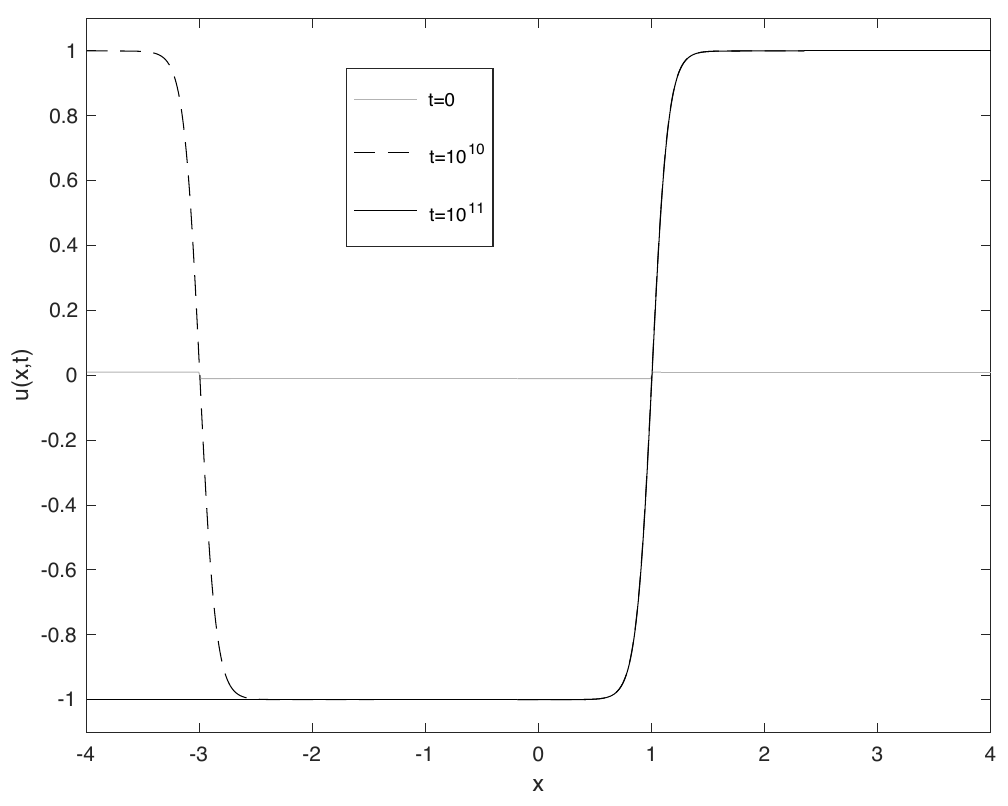}
		\hspace{3mm}
		\caption{\small{Here, we consider the same problem as in Figure \ref{Num1} with the discontinuous initial datum $u_0(x) = u_* \chi_{(-4,-3) \cup (1,4)}-u_* \chi_{(-3,1)}$, where $u_* =10^{-2}$. In the left hand side we depict the formation of a transition layer structure in relatively small times, while on the right hand side we see the subsequent slow motion.}}
		\label{Num3}
	\end{center}
\end{figure}

\subsubsection{Example no. 3} In this example we show what happens if considering a discontinuous initial datum $u_0$ which is a small perturbation of the unstable equilibrium zero. In the left hand side of Figure \ref{Num3} we can see how, in extremely short times, such configuration develops into a {\it continuous} function with two  transition layers; after that (right hand picture), we have to wait until $t \approx 10^{11}$ to see the first interface to collapse. In particular, such numerical experiment shows that, even if the solution does not satisfy  the assumptions \eqref{eq:ass-u0}-\eqref{eq:energy-ini} at $t=0$, at $t=9$ we already entered into the framework described by Theorem \ref{thm:main} and we witness the exponentially slow motion of the solution. This picture is not surprising, since it seems to confirm the well known behavior of the solution to the linear diffusion equation $u_t= \e^2 u_{xx} -F'(u)$; indeed, in \cite{Chen} the author rigorously proves that there are different phases in the dynamics, the first one being the generation of a metastable layered solution, which is governed by the ODE $u_t = -F'(u)$.
We conjecture that the same results hold true also in the nonlinear diffusion case \eqref{eq:Q-model}, being $Q(0)=0$.

\subsubsection{Example no. 4} In this last numerical experiment, we illustrate the results of Theorem \ref{thm:main2} by choosing $F$ as in \eqref{F:ex} with $\theta >2$. In both the pictures of Figure \ref{Num4} we select $Q(s)=se^{-s^2}$, and either the initial datum of Figure \ref{Num1} (left picture) or the discontinuous one of Figure \ref{Num3} (right picture), so that we can compare the two numerical experiments.
In the left hand picture, since $\theta=4$, the time taken to see  two bumps disappear is $t \approx 450$ (hence much smaller if compared to the right hand side of Figure \ref{Num1}, where one has to wait until $t \approx 10^4$). Similarly, in the right side of Figure \ref{Num4} (where $\theta=3$), the time employed by the  first interface to disappear is $t \approx 8*10^4 \ll 10^{11}$, that is the time exhibited in the right hand side of Figure \ref{Num3}.

\begin{figure}[t]
\begin{center}
\includegraphics[width=7.5cm,height=5.7cm]{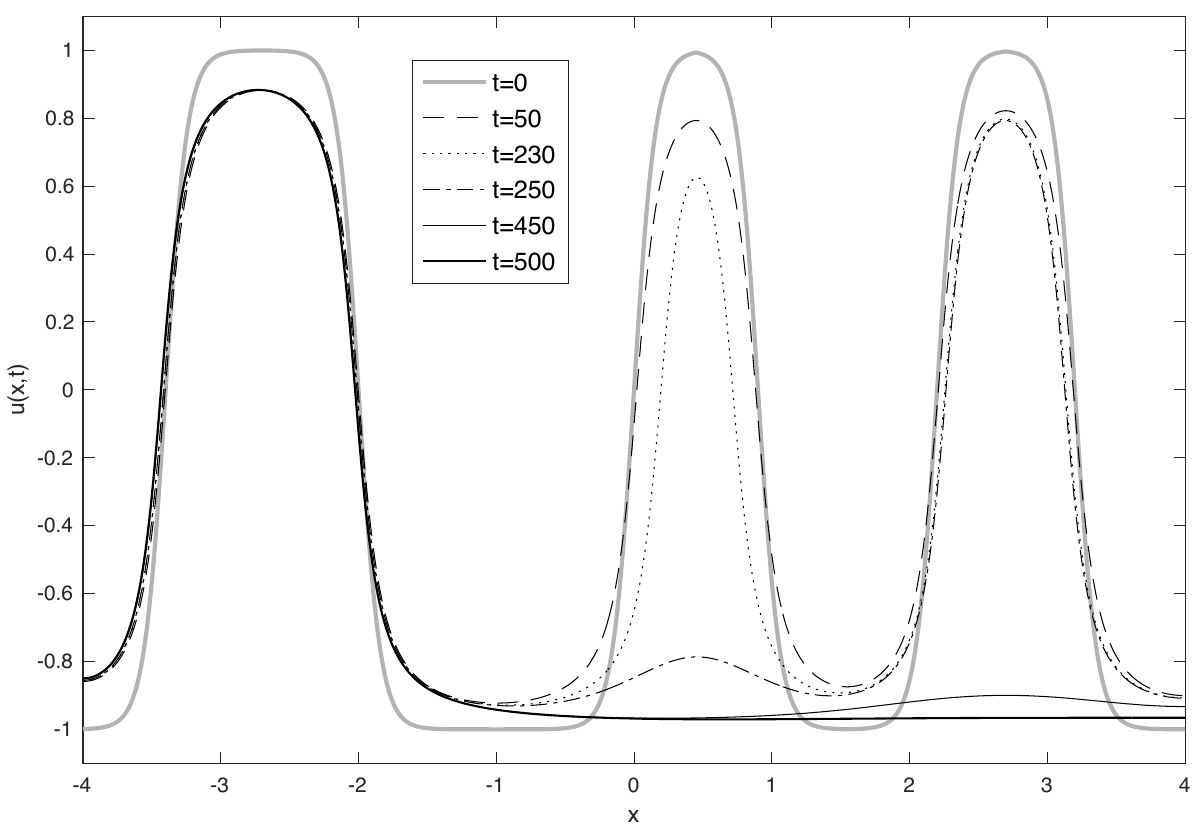}
\,
\includegraphics[width=7cm,height=5.7cm]{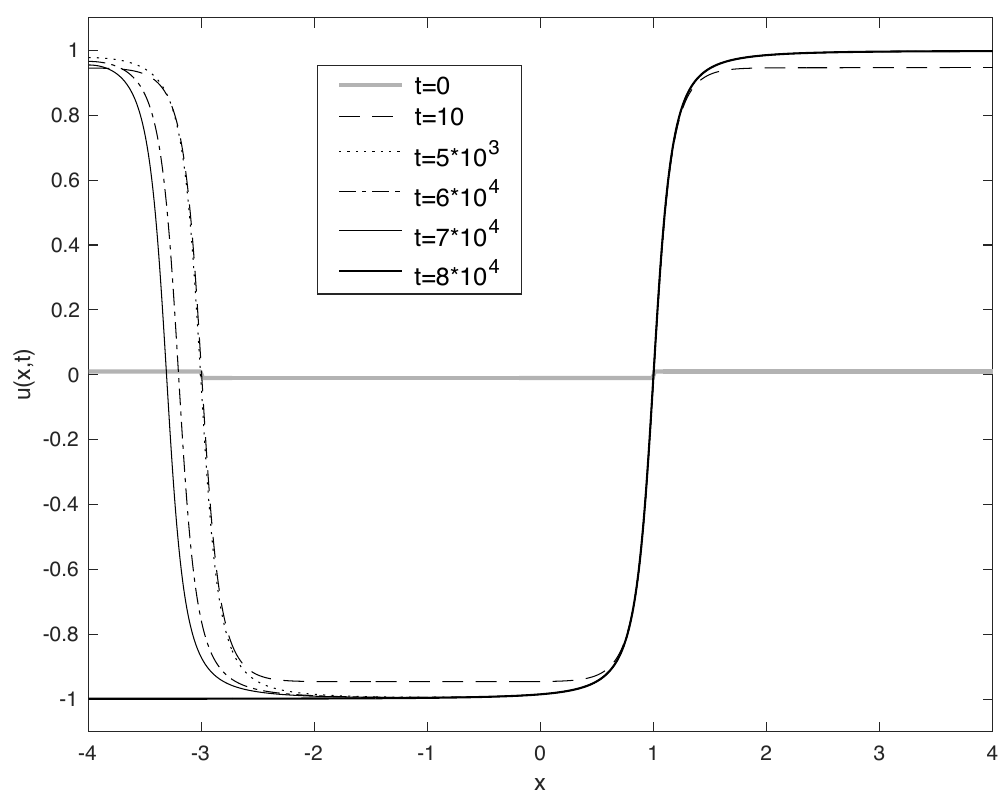}
\hspace{3mm}
\caption{\small{In this figure we depict the solution to \eqref{eq:Q-model}-\eqref{eq:Neu}-\eqref{eq:initial} with $Q(s)=s e^{-s^2}$, $\e=0.1$ and $F$ as in \eqref{F:ex} with $\theta=4$ (left picture) and $\theta=3$ (right picture). The initial datum $u_0$ is as in Figure \ref{Num1}} in the left picture, and as in Figure \ref{Num3} in the right picture.}
\label{Num4}
\end{center}
\end{figure}

\section*{Acknowledgements}
The work of R. Folino was partially supported by DGAPA-UNAM, program PAPIIT, grant IA-102423.


\begin{thebibliography}{99}

\bibitem{Allen-Cahn}
S. Allen and J. Cahn.
A microscopic theory for antiphase boundary motion and its application to antiphase domain coarsening.
{\it Acta Metall.}, {\bf 27} (1979), 1085--1095.

\bibitem{Bet-Sme}
F. Bethuel and D. Smets.
Slow motion for equal depth multiple-well gradient systems: the degenerate case.
{\it Discrete Contin. Dyn. Syst.}, {\bf33} (2013), 67--87.

\bibitem{Bron-Kohn}
L. Bronsard and R. Kohn.
On the slowness of phase boundary motion in one space dimension.
{\it Comm. Pure Appl. Math.}, {\bf 43} (1990), 983--997.

\bibitem{Carr-Pego}
J. Carr and R. L. Pego.
Metastable patterns in solutions of $u_t=\varepsilon^2u_{xx}-f(u)$.
{\it Comm. Pure Appl. Math.}, {\bf 42} (1989), 523--576.

\bibitem{Chen}
X. Chen.
Generation, propagation, and annihilation of metastable patterns.
{\it J. Differ. Equ.}, {\bf 206} (2004), 399--437.

\bibitem{CorMalSov}
A. Corli, L. Malaguti and E. Sovrano.
Wavefront solutions to reaction-convection equations with Perona-Malik diffusion.
{\it J. Differ. Equ.}, {\bf 208} (2022), 474--506.


\bibitem{Dra-Rob}
P. Dr\'abek and S. B. Robinson.
Continua of local minimizers in a non-smooth model of phase transitions.
{\it Z. Angew. Math. Phys.}, {\bf62} (2011), 609--622.

\bibitem{FPS}
R. Folino, R. G. Plaza and M. Strani. 
Metastable patterns for a reaction-diffusion model with mean curvature-type diffusion.
\emph{J. Math. Anal. Appl.}, {\bf493} (2021), article 124455.

\bibitem{FPS-DCDS}
R. Folino, R. G. Plaza and M. Strani. 
Long time dynamics of solutions to $p$-Laplacian diffusion problems with bistable reaction terms.
\emph{Discrete Contin. Dyn. Syst.}, {\bf41} (2021), 3211--3240.

\bibitem{FS}
R. Folino and M. Strani.
On reaction-diffusion models with memory and mean curvature-type diffusion.
\emph{J. Math. Anal. Appl.}, {\bf 522} (2023), article 127027.

\bibitem{FS-NA}
R. Folino and M. Strani.
On the speed rate of convergence of solutions to conservation laws with nonlinear diffusions.
{\it Nonlinear Analysis}, {\bf196} (2020), article 111762.


\bibitem{Fusco-Hale}
G. Fusco and J. Hale.
Slow-motion manifolds, dormant instability, and singular perturbations.
{\it J. Dynamics Differential Equations}, {\bf 1} (1989), 75--94.

\bibitem{Gobbi}
M. Gobbino.
Entire solutions of the one-dimensional Perona--Malik equation.
{\it Commun. Partial Differ. Equ.}, {\bf 32} (2007), 719--743.

\bibitem{Guido}
P. Guidotti.
A backward-forward regularization of the Perona--Malik equation.
{\it J. Differ. Equ.}, {\bf 252} (2012), 3226--3244.

\bibitem{Grant}
C. P. Grant.
Slow motion in one-dimensional Cahn--Morral systems.
{\it SIAM J. Math. Anal.}, {\bf 26} (1995), 21--34.

\bibitem{Hollig}
K. H\"ollig.
Existence of infinitely many solutions for a forward-backward heat equation. 
{\it Trans. Amer. Math. Soc.}, {\bf 278} (1983), 299--316.

\bibitem{Kaw-Kut}
B. Kawohl and N. Kutev.
Maximum and comparison principle for one-dimensional anisotropic diffusion.
{\it Math. Ann.}, {\bf 311}(1998), 107--123.

\bibitem{Kich}
S. Kichenassamy. 
The Perona--Malik paradox.
{\it SIAM J. Appl. Math.}, {\bf 57} (1997), 1328--1342.

\bibitem{Morfu}
S. Morfu.
On some applications of diffusion processes for image processing.
{\it Physics Letters A}, {\bf 373} (2009), 2438--2444.

\bibitem{PerMal}
P. Perona and J. Malik.
Scale-space and edge detection using anisotropic diffusion.
{\it IEEE Trans. Pattern Anal. and Machine Intell.}, {\bf12} (1990), 629--639.

\end{thebibliography}
\end{document}